\documentclass[11pt,reqno]{amsart}

\usepackage{color}
\usepackage{amsmath, amsthm, amssymb}
\usepackage{amsfonts}

\usepackage[ansinew]{inputenc}
\usepackage[dvips]{epsfig}
\usepackage{graphicx}
\usepackage[english]{babel}
\usepackage{hyperref}
\theoremstyle{plain}
\newtheorem{thm}{Theorem}[section]
\newtheorem{cor}[thm]{Corollary}
\newtheorem{lem}[thm]{Lemma}
\newtheorem{prop}[thm]{Proposition}

\theoremstyle{definition}
\newtheorem{defi}[thm]{Definition}

\theoremstyle{remark}
\newtheorem{rem}[thm]{Remark}

\numberwithin{equation}{section}

\newcommand{\average}{{\mathchoice {\kern1ex\vcenter{\hrule height.4pt
width 6pt depth0pt} \kern-9.7pt} {\kern1ex\vcenter{\hrule
height.4pt width 4.3pt depth0pt} \kern-7pt} {} {} }}
\newcommand{\ave}{\average\int}
\def\R{\mathbb{R}}
\def\N{\mathbb{N}}

\newcommand{\Per}{\mathrm{Per}}

\begin{document}

\title[Nonlocal problems with Neumann boundary conditions]{Nonlocal problems \\ with Neumann boundary conditions}

\author{Serena Dipierro}
\address{Maxwell Institute for Mathematical Sciences and School of Mathematics,
University of Edinburgh, James Clerk Maxwell Building, King's Buildings, Edinburgh EH9 3JZ,
United Kingdom}
\email{serena.dipierro@ed.ac.uk}

\author{Xavier Ros-Oton}
\address{Universitat Polit\`ecnica de Catalunya, Departament de Matem\`{a}tica  Aplicada I, Diagonal 647, 08028 Barcelona, Spain}
\email{xavier.ros.oton@upc.edu}

\author{Enrico Valdinoci}
\address{Weierstrass Institut f\"ur Angewandte Analysis und Stochastik, Mohrenstrasse 39, 10117 Berlin, Germany}
\email{enrico.valdinoci@wias-berlin.de}

\thanks{The first author was supported by grant EPSRC EP/K024566/1 (Scotland).
The second author was supported by grants MTM2011-27739-C04-01 (Spain) and 2009SGR345 (Catalunya).
The third author was supported by grants ERC-277749 (Europe)
and PRIN-201274FYK7 (Italy).
We thank Gerd Grubb for her very interesting comments
on a previous version of this manuscript.}

\begin{abstract}
We introduce a new Neumann problem for the fractional Laplacian arising from a simple probabilistic consideration,
and we discuss the basic properties of this model.
We can consider both elliptic and parabolic equations in any domain.
In addition, we formulate problems with nonhomogeneous Neumann conditions, and also with mixed Dirichlet and Neumann conditions, all of them having a clear probabilistic interpretation.

We prove that solutions to the fractional heat equation with homogeneous Neumann conditions have the following natural properties: conservation of mass inside $\Omega$, decreasing energy, and convergence to a constant as $t\to \infty$.
Moreover, for the elliptic case we give the variational formulation of the problem, and establish existence of solutions.

We also study the limit properties and the boundary behavior induced
by this nonlocal Neumann condition.

For concreteness, one may think that our nonlocal analogue of the classical
Neumann condition~$\partial_\nu u=0$ on~$\partial\Omega$ consists
in the nonlocal prescription
$$ \int_\Omega \frac{u(x)-u(y)}{|x-y|^{n+2s}}\,dy=0 \ {\mbox{ for }} x\in\R^n\setminus\overline{\Omega}.$$
We made an effort to keep all the arguments at the simplest
possible technical level, in order to clarify the conncetions
between the different scientific fields that are naturally involved in the
problem, and make the paper accessible
also to a wide, non-specialistic public (for this scope, we also
tried to use and compare different concepts and notations
in a somehow more unified way).
\end{abstract}

\maketitle

\section{Introduction and results}

The aim of this paper is to introduce the following Neumann problem for the fractional Laplacian
\begin{equation}\label{eq-intro}
\left\{ \begin{array}{rcll}
(-\Delta)^su&=&f&\textrm{in }\Omega \\
\mathcal N_su&=&0&\textrm{in }\R^n\setminus\overline{\Omega}.
\end{array}\right.
\end{equation}
Here, $\mathcal N_s$ is a new ``nonlocal normal derivative'', given by
\begin{equation}\label{normal-s}
\mathcal N_su(x):=c_{n,s}\int_\Omega \frac{u(x)-u(y)}{|x-y|^{n+2s}}\,dy, \qquad x\in \R^n\setminus\overline\Omega.
\end{equation}
The normalization constant $c_{n,s}$ is the one appearing in the definition the fractional Laplacian
\begin{equation}\label{operator}
(-\Delta)^s u(x)=c_{n,s}\,\textrm{PV}\int_{\R^n}\frac{u(x)-u(y)}{|x-y|^{n+2s}}\,dy.
\end{equation}
See \cite{DPV,S-obst} for the basic properties of this operator
(and for further details on the normalization constant~$c_{n,s}$,
whose explicit value only plays a minor role in this paper).

As we will see below, the corresponding heat equation with homogeneous Neumann conditions
\begin{equation}\label{eq-heat-intro}
\left\{ \begin{array}{rclll}
u_t +(-\Delta)^su&=&0&\textrm{in }\Omega,&\ t>0 \\
\mathcal N_su&=&0&\textrm{in }\R^n\setminus\overline{\Omega},&\ t>0\\
u(x,0)&=&u_0(x)&\textrm{in }\Omega,&\ t=0
\end{array}\right.
\end{equation}
possesses natural properties like conservation of mass inside $\Omega$ or convergence to a constant as $t\rightarrow+\infty$ (see Section~\ref{S4}).

The probabilistic interpretation of the Neumann problem \eqref{eq-heat-intro} may be summarized as follows:
\begin{enumerate}
\item $u(x,t)$ is the probability distribution of the position of a particle moving randomly inside $\Omega$.
\item When the particle exits $\Omega$, it immediately comes back into $\Omega$.
\item The way in which it comes back inside $\Omega$ is the following: If the particle has gone to $x\in\R^n\setminus\overline{\Omega}$, it may come back to any point $y\in\Omega$, the probability density
of jumping from $x$ to $y$ being proportional to $|x-y|^{-n-2s}$.
\end{enumerate}
These three properties lead to the equation \eqref{eq-heat-intro}, being $u_0$ the initial probability distribution of the position of the particle.

A variation of formula~\eqref{normal-s} consists in
renormalizing $\mathcal N_s u$ according to the underlying probability law induced by the L\'evy process. This leads to the definition
\begin{equation}\label{N-tilde}
\tilde{\mathcal N}_s u(x):=\frac{\mathcal N_s u(x)}{c_{n,s}\int_{\Omega}\frac{dy}{|x-y|^{n+2s}}}.
\end{equation}

Other Neumann problems for the fractional Laplacian (or other nonlocal operators) were introduced in \cite{BBC,CK}, \cite{BCGJ,BGJ}, \cite{CERW,CERW2,CERW3},  \cite{Grubb2}, and \cite{spectral,spectral2}.
All these different Neumann problems for nonlocal operators recover the classical Neumann problem as a
limit case, and most of them has clear probabilistic interpretations as well.
We postpone to Section~\ref{S1} a comparison between these different models and ours.

An advantage of our approach is that the problem has a variational structure.
In particular, we show that the classical integration by parts
formulae
\begin{eqnarray*}
&& \int_\Omega \Delta u=\int_{\partial\Omega} \partial_\nu u\\
{\mbox{and }}&&
\int_\Omega \nabla u\cdot\nabla v=\int_\Omega v\,(-\Delta)u+
\int_{\partial\Omega} v\partial_\nu u
\end{eqnarray*}
are replaced in our setting by
\[\int_\Omega (-\Delta)^su\,dx=-\int_{\R^n\setminus\Omega} \mathcal N_su\,dx\]
and
\[\frac{c_{n,s}}{2}\int_{\R^{2n}\setminus(\mathcal C\Omega)^2}\frac{\bigl(u(x)-u(y)\bigr)\bigl(v(x)-v(y)\bigr)}{|x-y|^{n+2s}}\,dx\,dy
=\int_\Omega v\,(-\Delta)^su+\int_{\R^n\setminus\Omega} v\,\mathcal N_su.\]
Also, the classical Neumann problem
\begin{equation}\label{CL}
\left\{ \begin{array}{rcll}
-\Delta u&=&f&\textrm{in }\Omega \\
\partial_\nu u&=&g &\textrm{on }\partial\Omega
\end{array}\right.
\end{equation}
comes from critical points of the energy functional
$$ \frac12 \int_\Omega |\nabla u|^2-\int_\Omega fu-\int_{\partial\Omega} g\,u,$$
without trace conditions. In analogy with this, we
show that our nonlocal Neumann condition
\begin{equation}\label{NLC}
\left\{ \begin{array}{rcll}
(-\Delta)^su&=&f&\textrm{in }\Omega \\
\mathcal N_su&=&g&\textrm{in }\R^n\setminus\overline{\Omega}
\end{array}\right.
\end{equation}
follows
from free critical points of the energy functional
$$ \frac{c_{n,s}}{4}\int_{\R^{2n}\setminus(\mathcal C\Omega)^2}\frac{|u(x)-u(y)|^2}{|x-y|^{n+2s}}\,dx\,dy-\int_\Omega f\,u-\int_{\R^n\setminus\Omega} g\,u,$$
see Proposition~\ref{prop free}.
Moreover, as well known, the theory of existence and uniqueness
of solutions for the classical Neumann problem~\eqref{CL} relies on the
compatibility condition
$$ \int_\Omega f =-\int_{\partial\Omega}g.$$
We provide the analogue of this compatibility condition
in our framework, that is
$$ \int_\Omega f=-\int_{\R^n\setminus\Omega} g,$$
see Theorem~\ref{EX}.
Also, we give a description of the spectral properties
of our nonlocal problem, which are in analogy with the classical case.
\medskip

The paper is organized in this way.
In Section~\ref{sec-prob}
we give a probabilistic interpretation of our Neumann condition,
as a random reflection of a particle inside the domain, according to
a L\'evy flight. This also allows us to consider mixed Dirichlet and
Neumann conditions and to get a suitable heat equation from the stochastic process.

In Section~\ref{S3} we consider the variational structure
of the associated nonlocal elliptic problem, we show an existence and uniqueness result
(namely Theorem~\ref{EX}), as follows:

\bigskip

\noindent{\it Let $\Omega\subset\R^n$ be a bounded Lipschitz domain, $f\in L^2(\Omega)$,
and $g\in L^1(\R^n\setminus\Omega)$.
Suppose that there exists a $C^2$ function~$\psi$
such that~$\mathcal{N}_s\psi=g$ in~$\R^n\setminus\overline{\Omega}$.

Then, problem \eqref{NLC} admits a weak solution
if and only if
$$\int_\Omega f=-\int_{\R^n\setminus\Omega} g.
$$
Moreover, if such a compatibility condition holds,
the solution is unique up to an additive constant.
}
\bigskip

Also, we give a description of a sort of generalized
eigenvalues of~$(-\Delta)^s$ with zero
Neumann boundary conditions (see Theorem~\ref{EVAL}):
\bigskip

\noindent{\it
Let $\Omega\subset\R^n$ be a bounded Lipschitz domain.
Then, there exist a sequence of nonnegative values
\[0=\lambda_1<\lambda_2\leq \lambda_3\leq\cdots,\] and
a sequence of functions~$u_i:\R^n\to\R$ such that
\[\left\{ \begin{array}{rcll}
(-\Delta)^su_i(x)&=&\lambda_i u_i(x)&\textrm{for any }x\in\Omega \\ \mathcal
N_su_i(x)&=&0&\textrm{for any }x\in\R^n\setminus\overline{\Omega}.
\end{array}\right.\]
Also, the functions~$u_i$ (when restricted to~$\Omega$)
provide a complete orthogonal system in~$L^2(\Omega)$.
}
\bigskip

By similarity with the classical case, we are tempted
to consider the above~$\lambda_i$ and~$u_i$ as generalized
eigenvalues and eigenfunctions. Though the word ``generalized''
will be omitted from now on for the sake of shortness,
we remark that this spectral notion is not completely standard,
since our eigenfunctions~$u_i$ are defined in the whole of~$\R^n$ but
satisfy the equation~$(-\Delta)^s u_i=\lambda_i u_i$
only in the domain~$\Omega$ (indeed, outside $\Omega$ they verify our
nonlocal Neumann condition). Moreover, the orthogonality
and density properties of~$u_i$ also refer to their restriction in~$\Omega$.
\medskip

In Section~\ref{S4} we discuss the associated heat equation.
As it happens in the classical case, we show that such equation
preserves the mass, it has decreasing energy, and the solutions
approach a constant as $t\to+\infty$. In particular, by the results in
Propositions~\ref{mass}, \ref{eneergy} and~\ref{89} we have:
\bigskip

\noindent{\it
Assume that $u(x,t)$ is a classical solution to
\begin{equation*}
\left\{ \begin{array}{rclll}
u_t +(-\Delta)^su&=&0&\textrm{in }\Omega,&\ t>0 \\
\mathcal N_su&=&0&\textrm{in }\R^n\setminus\overline{\Omega},&\ t>0\\
u(x,0)&=&u_0(x)&\textrm{in }\Omega,&\ t=0.
\end{array}\right.
\end{equation*}
Then the total mass is conserved, i.e. for all $t>0$
\[\int_\Omega u(x,t)\,dx=\int_\Omega u_0(x)dx.\]
Moreover, the energy
\[E(t)=\int_{\R^{2n}\setminus(\mathcal C \Omega)^2} \frac{|u(x,t)-u(y,t)|^2}{|x-y|^{n+2s}}\,dx\,dy\]
is decreasing in time $t>0$.

Finally, the solution approaches a constant for large times: more precisely
\[u\,\longrightarrow\,
\frac{1}{|\Omega|}\int_\Omega u_0\quad \textrm{in}\ L^2(\Omega)\]
as $t\rightarrow+\infty$.
}
\bigskip

In Section~\ref{Slimit} we compute some limits when~$s\to1$,
showing that we can recover the classical case.  In particular,
we show in Proposition~\ref{89X} that:
\bigskip

\noindent{\it
Let $\Omega\subset\R^n$ be any bounded Lipschitz domain.
Let $u$ and $v$ be $C^2_0(\R^n)$ functions.
Then,
\[\lim_{s\rightarrow 1}\int_{\R^n\setminus\Omega}\mathcal N_s u\,v=\int_{\partial\Omega}\frac{\partial u}{\partial\nu}\,v.\]
}
\bigskip

Also, we prove that nice functions can be extended
continuously outside~$\overline\Omega$ in order to satisfy
a homogeneous nonlocal Neumann condition, and we characterize
the boundary behavior of the nonlocal Neumann function.
More precisely, in Proposition~\ref{cont} we show that:
\bigskip

\noindent{\it
Let~$\Omega\subset\R^n$ be a domain with~$C^1$ boundary.
Let~$u$ be continuous in~$\overline\Omega$, with~${\mathcal{N}}_s u=0$
in~$\R^n\setminus\overline\Omega$. Then~$u$ is continuous
in the whole of~$\R^n$.
}
\bigskip

The boundary behavior of the nonolcal Neumann condition is also addressed in Proposition~\ref{89Y}:
\bigskip

\noindent{\it Let $\Omega\subset\R^n$ be a $C^1$ domain, and $u\in C(\R^n)$.
Then, for all $s\in(0,1)$,
$$\lim_{{x\rightarrow\partial\Omega}\atop{x\in\R^n\setminus\overline\Omega}}\tilde{\mathcal N}_s u(x)=0,$$
where $\tilde{\mathcal N}$ is defined by \eqref{N-tilde}.

Also, if $s>\frac12$ and~$u\in C^{1,\alpha}(\R^n)$ for some~$\alpha>0$, then
\begin{equation*}\partial_\nu\tilde{\mathcal N}_s u(x):=
\lim_{\epsilon\to0^+} \frac{\tilde{\mathcal N}_s u(x+\epsilon\nu)
}{\epsilon}=\kappa \, \partial_\nu u\qquad \textrm{for any}\ x\in\partial\Omega,
\end{equation*}
for some constant $\kappa>0$.
}
\bigskip

Later on, in Section~\ref{sover} we deal with an overdetermined problem and we show that
it is not possible to prescribe both nonlocal Neumann and Dirichlet conditions for a continuous function.

Finally, in Section~\ref{S1}
we recall the various nonlocal Neumann conditions already appeared
in the literature, and we compare them with our model.

All the arguments presented are of elementary\footnote{\label{FT1}To keep the notation
as simple as possible, given functions~$f$ and~$g$
and an operator~$T$, we will often write idendities like~``$f=g$ in $\Omega$''
or~``$Tf=g$ in $\Omega$'' to mean ``$f(x)=g(x)$ for every~$x\in\Omega$''
or~``$Tf=g$ for every~$x\in\Omega$'', respectively. Also, if~$u:\R^n\to\R$, we often denote
the restriction of~$u$ to~$\Omega$ again by~$u$.
We hope that this slight abuse of notation creates no problem to
the reader, but for the sake of clarity we also include an appendix
at the end of the paper
in which Theorems~\ref{EX} and~\ref{EVAL} are proved
using a functional analysis notation that distinguishes between
a function and its restriction.}
nature.

\section{Heuristic probabilistic interpretation}
\label{sec-prob}

In this section we will give a simple probabilistic interpretation
of the nonlocal Neumann condition that we consider
in terms of the so-called L\'evy flights. Though
the possible behavior of a general L\'evy process can be
more sophisticated than the one we consider, for the sake of
clarity we will try to restrict ourselves to the simplest
possible scenario and to use the simplest possible language.
For this scope, we will not go into all the very rich details
of the related probability theory and we will not aim to review
all the important, recent results on the topic, but we will rather
present an elementary, self-contained exposition, which we hope can
serve as an introduction also to a non-specialistic public.\medskip

Let us consider the L\'evy process in $\R^n$ whose infinitesimal generator is the fractional Laplacian $(-\Delta)^s$.
Heuristically, we may think that this process represents the (random) movement of a particle along time $t>0$.
As it is well known, the probability density $u(x,t)$ of the position of the particle solves the fractional heat equation $u_t+(-\Delta)^su=0$ in $\R^n$; see \cite{V-SEMA} for a simple illustration of this fact.

Recall that when the particle is situated at $x\in\R^n$, it may jump to any other point $y\in\R^n$, the probability density
of jumping to $y$ being proportional to $|x-y|^{-n-2s}$.

Similarly, one may consider the random movement of a particle inside a bounded domain
$\Omega\subset \R^n$, but in this case one has to decide what happens when the
particle leaves $\Omega$.

In the classical case $s=1$ (when the L\'evy process is the Brownian motion), we have the following:
\begin{enumerate}
\item If the particle is \emph{killed} when it reaches the boundary $\partial\Omega$, then the probability distribution solves the heat equation with homogeneous \emph{Dirichlet} conditions.
\item If, instead, when the particle reaches the boundary $\partial\Omega$ it immediately \emph{comes back into} $\Omega$ (i.e., it bounces on $\partial\Omega$), then the probability distribution solves the heat equation with homogeneous \emph{Neumann} conditions.
\end{enumerate}

In the nonlocal case $s\in(0,1)$, in which the process has jumps,
case (1) corresponds to the following: The particle is killed when it exits $\Omega$.
In this case, the probability distribution $u$ of the process solves
the heat equation with homogeneous Dirichlet conditions $u=0$ in $\R^n\setminus\Omega$,
and solutions to this problem are well understood; see for example \cite{RS-Dir,Grubb,FKV,BCI2}.

The analogue of case (2) is the following: When the particle exits $\Omega$, it immediately comes back into $\Omega$.
Of course, one has to decide how the particle comes back into the domain.

In \cite{BCGJ,BGJ} the idea was to find a deterministic ``reflection'' or ``projection'' which describes the way in which the particle comes back into $\Omega$.

The alternative that we propose here is the following:
If the particle has gone to $x\in\R^n\setminus\overline{\Omega}$, then it may come back to \emph{any} point $y\in\Omega$, the probability density of
jumping from $x$ to $y$ being proportional to $|x-y|^{-n-2s}$.

Notice that this is exactly the (random) way as the particle is moving all the time, here we just add the restriction that it has to immediately come back into $\Omega$ every time it goes outside.

Let us finally illustrate how this random process leads
to problems \eqref{eq-intro} or \eqref{eq-heat-intro}.
In fact, to make the exposition easier, we will explain the case of mixed Neumann and
Dirichlet conditions, which, we think, is very natural.

\subsection{Mixed Dirichlet and Neumann conditions}

Assume that we have some domain $\Omega\subset\R^n$, and that its complement $\R^n\setminus\overline{\Omega}$ is splitted into two parts: $N$ (with Neumann conditions), and $D$ (with Dirichlet conditions).

Consider a particle moving randomly, starting inside $\Omega$.
When the particle reaches $D$, it obtains a payoff $\phi(x)$, which depends on the point $x\in D$ where the particle arrived.
Instead, when the particle reaches $N$ it immediately comes back to $\Omega$ as described before.

If we denote $u(x)$ the expected payoff, then we clearly have
\[(-\Delta)^su=0\quad \textrm{in}\ \Omega\]
and
\[u=\phi\quad \textrm{in}\ D,\]
where $\phi:D\longrightarrow \R$ is a given function.

Moreover, recall that when the particle is in $x\in N$ then it goes back to some point $y\in\Omega$, with probability proportional to $|x-y|^{-n-2s}$.
Hence, we have that
\[u(x)=\kappa\int_\Omega \frac{u(y)}{|x-y|^{n+2s}}\,dy\qquad \textrm{for}\ x\in N,\]
for some constant $\kappa$, possibly depending on the point~$x\in N$, that has been fixed.
In order to normalize the probability measure,
the value of the constant $\kappa$ is so that
$$ \kappa\int_\Omega \frac{dy}{|x-y|^{n+2s}}=1.$$

Finally, the previous identity can be written as
\[\mathcal N_su(x)=c_{n,s}\int_\Omega \frac{u(x)-u(y)}{|x-y|^{n+2s}}\,dy=0\quad \textrm{for}\ x\in N,\]
and therefore $u$ solves
\[\left\{ \begin{array}{rcll}
(-\Delta)^su&=&0&\textrm{in }\Omega \\
\mathcal N_su&=&0&\textrm{in }N\\
u&=&\phi&\textrm{in }D,
\end{array}\right.\]
which is a nonlocal problem with mixed Neumann and Dirichlet conditions.

Note that the previous problem is the nonlocal analogue of
\[\left\{ \begin{array}{rcll}
-\Delta u&=&0&\textrm{in }\Omega \\
\partial_\nu u&=&0&\textrm{in }\Gamma_N\\
u&=&\phi&\textrm{in }\Gamma_D,
\end{array}\right.\]
being $\Gamma_D$ and $\Gamma_N$ two disjoint subsets of $\partial\Omega$,
in which classical Dirichlet and Neumann boundary conditions are prescribed.

More generally, the classical Robin condition $a\partial_\nu u+b u=c$
on some~$\Gamma_R\subseteq \partial\Omega$
may be replaced in our nonlocal framework by~$a\mathcal N_s u+b u=c$
on some~$R\subseteq \R^n\setminus\overline{\Omega}$. Nonlinear boundary
conditions may be considered in a similar way.

\subsection{Fractional heat equation, nonhomogeneous Neumann conditions}

Let us consider now the random movement of the particle inside $\Omega$, with our new Neumann conditions in $\R^n\setminus\overline{\Omega}$.

Denoting $u(x,t)$ the probability density of the position of the particle at time $t>0$, with a similar discretization argument as in \cite{V-SEMA}, one can see that $u$ solves the fractional heat equation
\[u_t+(-\Delta)^su=0\quad \textrm{ in }\Omega\quad\, \textrm{for}\ t>0,\]
with
\[\mathcal N_su=0\quad \textrm{ in }\R^n\setminus\overline{\Omega}\quad\, \textrm{for}\ t>0.\]
Thus, if $u_0$ is the initial probability density, then $u$ solves problem \eqref{eq-heat-intro}.

Of course, one can now see that with this probabilistic interpretation there is no problem in considering a right hand side $f$ or nonhomogeneous Neumann conditions
\[\left\{ \begin{array}{rcll}
u_t+(-\Delta)^su&=&f(x,t,u)&\textrm{in }\Omega \\
\mathcal N_su&=&g(x,t)&\textrm{in }\R^n\setminus\overline{\Omega}.
\end{array}\right.\]
In this case, $g$ represents a ``nonlocal flux'' of new particles coming from outside $\Omega$, and $f$ would represent a reaction term.

\section{The elliptic problem}\label{S3}

Given $g\in L^1(\R^n\setminus\Omega)$ and measurable
functions~$u,v:\R^n\rightarrow\R$, we set
\begin{equation}\begin{split}\label{norm}
\|u\|_{H^s_{\Omega,g}}&:= \sqrt{ \|u\|_{L^2(\Omega)}^2+
\||g|^{1/2}\,u\|_{L^2(\R^n\setminus\Omega)}^2
+\int_{\R^{2n}\setminus(\mathcal C\Omega)^2}
\frac{|u(x)-u(y)|^2}{|x-y|^{n+2s}}\,dx\,dy }
\end{split}\end{equation}
and
\begin{equation}\begin{split}\label{scalar}
(u,v)_{H^s_{\Omega,g}}& := \int_{\Omega}u\,v\, dx+ \int_{\R^n\setminus\Omega}|g|\,u\, v\, dx \\
&\qquad + \int_{\R^{2n}\setminus(\mathcal C\Omega)^2}\frac{(u(x)-u(y))(v(x)-v(y))}{|x-y|^{n+2s}}
\,dx\,dy.
\end{split}\end{equation}
Then, we define the space
\[H^s_{\Omega,g}:=\left\{u:\R^n\rightarrow\R \ {\mbox{ measurable }}\ :\,
\|u\|_{H^s_{\Omega,g}}<+\infty\right\}.\]
We will also write~$H^s_{\Omega,0}$ to mean~$H^s_{\Omega,g}$
with $g\equiv0$.

\begin{prop}
$H^s_{\Omega,g}$ is a Hilbert space with the scalar product defined in~\eqref{scalar}.
\end{prop}

\begin{proof} We point out that~\eqref{scalar} is a bilinear form
and~$ \|u\|_{H^s_{\Omega,g}}=\sqrt{(u,u)_{H^s_{\Omega,g}}}$. Also,
if~$\|u\|_{H^s_{\Omega,g}}=0$, it follows that~$\|u\|_{L^2(\Omega)}=0$,
hence~$u=0$ a.e. in~$\Omega$, and that
$$ \int_{\R^{2n}\setminus(\mathcal C\Omega)^2}
\frac{|u(x)-u(y)|^2}{|x-y|^{n+2s}}\,dx\,dy=0,$$
which in turn implies that~$|u(x)-u(y)|=0$ for any~$(x,y)\in
\R^{2n}\setminus(\mathcal C\Omega)^2$. In particular, a.e.~$x\in\mathcal C\Omega$
and~$y\in\Omega$ we have that
$$ u(x)=u(x)-u(y)=0.$$
This shows that~$u=0$ a.e. in $\R^n$, so it remains to prove that~$H^s_{\Omega,g}$
is complete. For this, we take a Cauchy sequence~$u_k$
with respect to the norm in~\eqref{norm}.

In particular, $u_k$ is a Cauchy sequence in~$L^2(\Omega)$
and therefore, up to a subsequence, we suppose that $u_k$ converges
to some~$u$ in~$L^2(\Omega)$ and a.e. in~$\Omega$. More explicitly,
there exists~$Z_1\subset\R^n$ such that
\begin{equation}\label{H2}
{\mbox{$|Z_1|=0$ and $u_k(x)\to u(x)$ for every $x\in\Omega\setminus Z_1$.}}
\end{equation}
Also, given any~$U:\R^n\rightarrow\R$, for any~$(x,y)\in\R^{2n}$ we define
\begin{equation}\label{H7}
E_U(x,y):=\frac{\Big(U(x)-U(y)\Big)\,
\chi_{\R^{2n}\setminus(\mathcal C\Omega)^2}(x,y)}{|x-y|^{\frac{n+2s}{2}}}.\end{equation}
Notice that
$$ E_{u_k}(x,y)-E_{u_h}(x,y)=
\frac{\Big(u_k(x)-u_h(x)-u_k(y)+u_h(y)\Big)\,
\chi_{\R^{2n}\setminus(\mathcal C\Omega)^2}(x,y)}{|x-y|^{\frac{n+2s}{2}}}.$$
Accordingly, since~$u_k$ is a Cauchy sequence in~$H^s_{\Omega,g}$,
for any~$\epsilon>0$ there exists~$N_\epsilon\in\N$ such that
if~$h$, $k\ge N_\epsilon$ then
$$\epsilon^2\ge \int_{\R^{2n}\setminus(\mathcal C\Omega)^2}
\frac{|(u_k-u_h)(x)-(u_k-u_h)(y)|^2}{|x-y|^{n+2s}}\,dx\,dy
=\| E_{u_k}-E_{u_h}\|^2_{L^2(\R^{2n})}.$$
That is, $E_{u_k}$ is a Cauchy sequence in $L^2(\R^{2n})$
and thus, up to a subsequence, we assume that~$E_{u_k}$
converges to some~$E$ in~$L^2(\R^{2n})$ and a.e. in~$\R^{2n}$.
More explicitly, there exists~$Z_2\subset\R^{2n}$ such that
\begin{equation}\label{H0}
{\mbox{$|Z_2|=0$ and $E_{u_k}(x,y)\to E(x,y)$ for every $(x,y)\in\R^{2n}
\setminus Z_2$.}}
\end{equation}
Now, for any $x\in\Omega$, we set
\begin{eqnarray*} && S_x:=\{ y\in\R^n\ :\, (x,y)\in \R^{2n}\setminus Z_2\},\\
&& W:=\{(x,y)\in \R^{2n} \ :\, x\in\Omega {\mbox{ and }}
y\in\R^n\setminus S_x\}\\
{\mbox{and }}&& V:=\{ x\in\Omega \ :\, |\R^n\setminus S_x|=0\}.
\end{eqnarray*}
We remark that
\begin{equation}\label{H200}
W\subseteq Z_2.
\end{equation}
Indeed, if $(x,y)\in W$, then~$y\in\R^n\setminus S_x$,
hence~$(x,y)\not\in \R^{2n}\setminus Z_2$, and so~$(x,y)\in Z_2$,
which gives~\eqref{H200}.

Using~\eqref{H0} and~\eqref{H200}, we obtain that~$|W|=0$,
hence by the Fubini's Theorem we have that
$$ 0=|W| =\int_\Omega |\R^n\setminus S_x|\,dx,$$
which implies that~$|\R^n\setminus S_x|=0$ for a.e.~$x\in\Omega$.

As a consequence, we conclude that~$|\Omega\setminus V|=0$.
This and~\eqref{H2} imply that
$$ |\Omega\setminus (V\setminus Z_1)|
=|(\Omega\setminus V)\cup Z_1|\le |\Omega\setminus V|+|Z_1|=0.$$
In particular~$V\setminus Z_1\ne\varnothing$, so we can fix~$x_0
\in V\setminus Z_1$.

Since~$x_0\in \Omega\setminus Z_1$, equation~\eqref{H2} implies
$$ \lim_{k\to+\infty} u_k(x_0)=u(x_0).$$
Furthermore, since~$x_0\in V$
we have that~$|\R^n\setminus S_{x_0}|=0$. As a consequence, a.e. $y\in\R^n$
(namely, for every~$y\in S_{x_0}$),
we have that~$(x_0,y)\in\R^{2n}\setminus Z_2$ and so
$$ \lim_{k\to+\infty} E_{u_k}(x_0,y)= E(x_0,y),$$
thanks to~\eqref{H0}.
Notice also that~$\Omega\times(\mathcal C \Omega)\subseteq
\R^{2n}\setminus(\mathcal C\Omega)^2$ and so, recalling~\eqref{H7}, we get
$$
E_{u_k}(x_0,y):=\frac{u_k(x_0)-u_k(y)}{|x_0-y|^{\frac{n+2s}{2}}},$$
for a.e. $y\in\mathcal C\Omega$.
Thus, we obtain
\begin{eqnarray*}
\lim_{k\to+\infty}u_k(y)&=&
\lim_{k\to+\infty} \left\{u_k(x_0)-|x_0-y|^{\frac{n+2s}{2}}E_{u_k}(x_0,y)\right\} \\
&=& u(x_0)-|x_0-y|^{\frac{n+2s}{2}}E(x_0,y),
\end{eqnarray*}
a.e.~$y\in\mathcal C\Omega$.

This and~\eqref{H2} say that~$u_k$ converges a.e. in~$\R^n$.
Up to a change of notation, we will say that~$u_k$ converges a.e.
in~$\R^n$ to some~$u$. So, using that~$u_k$ is a Cauchy sequence
in~$H^s_{\Omega,g}$, fixed any~$\epsilon>0$ there exists~$N_\epsilon\in\N$
such that, for any~$h\ge N_\epsilon$,
\begin{eqnarray*}
\epsilon^2 &\ge& \liminf_{k\to+\infty}\|u_h-u_k\|_{H^s_{\Omega,g}}^2 \\
&\ge& \liminf_{k\to+\infty} \int_\Omega (u_h-u_k)^2
+\liminf_{k\to+\infty} \int_{\mathcal C\Omega} |g|(u_h-u_k)^2\\ &&\qquad+
\liminf_{k\to+\infty}
\int_{\R^{2n}\setminus(\mathcal C\Omega)^2}
\frac{|(u_h-u_k)(x)-(u_h-u_k)(y)|^2}{|x-y|^{n+2s}}\,dx\,dy
\\
&\ge& \int_\Omega (u_h-u)^2
+\int_{\mathcal C\Omega} |g|(u_h-u)^2+
\int_{\R^{2n}\setminus(\mathcal C\Omega)^2}
\frac{|(u_h-u)(x)-(u_h-u)(y)|^2}{|x-y|^{n+2s}}\,dx\,dy
\\ &=&
\|u_h-u\|_{H^s_{\Omega,g}}^2,
\end{eqnarray*}
where Fatou's Lemma was used. This says that~$u_h$ converges to~$u$
in~$H^s_{\Omega,g}$, showing that~$H^s_{\Omega,g}$ is complete.
\end{proof}

\subsection{Some integration by parts formulas}

The following is a nonlocal analogue of the divergence theorem.

\begin{lem}\label{lema1}
Let $u$ be any bounded $C^2$ function in $\R^n$.
Then,
\[\int_\Omega (-\Delta)^su=-\int_{\R^n\setminus\Omega} \mathcal N_su.\]
\end{lem}

\begin{proof}
Note that
\[\int_\Omega\int_\Omega \frac{u(x)-u(y)}{|x-y|^{n+2s}}\,dx\, dy=\int_\Omega\int_\Omega \frac{u(y)-u(x)}{|x-y|^{n+2s}}\,dx\, dy=0,\]
since the role of~$x$ and~$y$ in the integrals above is symmetric.
Hence, we have that
\[\begin{split}
\int_\Omega (-\Delta)^su\,dx&=c_{n,s}\int_\Omega \int_{\R^n}\frac{u(x)-u(y)}{|x-y|^{n+2s}}\,dy\, dx=c_{n,s}\int_\Omega \int_{\R^n\setminus\Omega}\frac{u(x)-u(y)}{|x-y|^{n+2s}}\,dy\, dx\\
&= c_{n,s}\int_{\R^n\setminus\Omega}\int_\Omega\frac{u(x)-u(y)}{|x-y|^{n+2s}}\,dx\, dy=-\int_{\R^n\setminus\Omega} \mathcal N_su(y)\,dy,
\end{split}\]
as desired.
\end{proof}

More generally, we have the following integration by parts formula.

\begin{lem}\label{lema2}
Let $u$ and $v$ be bounded $C^2$ functions in $\R^n$.
Then,
\[\frac{c_{n,s}}{2}\int_{\R^{2n}\setminus(\mathcal C\Omega)^2}\frac{\bigl(u(x)-u(y)\bigr)
\bigl(v(x)-v(y)\bigr)}{|x-y|^{n+2s}}\,dx\,dy
=\int_\Omega v\,(-\Delta)^su+\int_{\R^n\setminus\Omega} v\,\mathcal N_su\,,\]
where $c_{n,s}$ is the constant in \eqref{operator}.
\end{lem}

\begin{proof}
Notice that
\[\begin{split}
&\frac12\int_{\R^{2n}\setminus(\mathcal C\Omega)^2}\frac{\bigl(u(x)-u(y)\bigr)\bigl(v(x)-v(y)\bigr)}{|x-y|^{n+2s}}\,dx\,dy\\
&\qquad =\int_{\Omega}\int_{\R^n}v(x)\,\frac{u(x)-u(y)}{|x-y|^{n+2s}}\,dy\,dx +\int_{\R^n\setminus\Omega}\int_\Omega v(x)\,\frac{u(x)-u(y)}{|x-y|^{n+2s}}\,dy\,dx.\end{split}\]
Thus, using \eqref{operator} and \eqref{normal-s}, the identity follows.
\end{proof}

\begin{rem}\label{rem1}
We recall that if one takes $\partial_\nu u=1$,
then one can obtain the perimeter of $\Omega$ by integrating
this Neumann condition over~$\partial\Omega$. Indeed,
\begin{equation}\label{per local}
|\partial\Omega|=\int_{\partial\Omega}dx= \int_{\partial\Omega}\partial_\nu u\, dx.
\end{equation}
Analogously, we can define $\tilde{\mathcal N}_s u$,
by renormalizing $\mathcal N_s u$ by a factor
$$ w_{s, \Omega}(x):= c_{n,s}\int_{\Omega}\frac{dy}{|x-y|^{n+2s}}, $$
that is
\begin{equation}\label{renormalized} \tilde{\mathcal N}_s u(x):=\frac{\mathcal N_s u(x)}{w_{s,\Omega}(x)} \
{\mbox{ for }}\ x\in\R^n\setminus\overline{\Omega}. \end{equation}
Now, we observe that if $\tilde{\mathcal N}_su(x)=1$ for any~$x\in\R^n\setminus\overline{\Omega}$, then
we find the fractional perimeter (see~\cite{CRS} where this object was introduced)
by integrating such nonlocal Neumann condition over~$\R^n\setminus\Omega$, that is:
\begin{eqnarray*}
\Per_s(\Omega)&:=& c_{n,s}\int_{\Omega}\int_{\R^n\setminus\Omega}\frac{dx\,dy}{|x-y|^{n+2s}}\\
&=& \int_{\R^n\setminus\Omega}w_{s, \Omega}(x)\,dx \\
&=& \int_{\R^n\setminus\Omega}w_{s, \Omega}(x)\,\tilde{\mathcal N}_s u(x)\, dx \\
&=& \int_{\R^n\setminus\Omega}\mathcal N_s u(x)\, dx,
\end{eqnarray*}
that can be seen as the nonlocal counterpart of~\eqref{per local}.
\end{rem}

\begin{rem}\label{rem2}
The renormalized Neumann condition
in~\eqref{renormalized} can also be framed into the probabilistic
interpretation of Section~\ref{sec-prob}.

Indeed suppose that $\mathcal C\Omega$
is partitioned into a Dirichlet part~$D$ and a Neumann part~$N$
and that:
\begin{itemize}
\item our L\'evy process receives a final payoff~$\phi(x)$
when it leaves the domain~$\Omega$ by landing at the point~$x$ in~$D$,
\item if the L\'evy process leaves $\Omega$ by landing at the point~$x$
in~$N$, then it receives an additional payoff~$\psi(x)$
and is forced to come back to~$\Omega$ and keep running by following the same probability law
(the case discussed in Section~\ref{sec-prob} is the model situation
in which~$\psi\equiv0$).
\end{itemize}
In this setting, the expected payoff~$u(x)$ obtained by starting the
process at the point~$x\in\Omega$ satisfies~$(-\Delta)^s u=0$ in~$\Omega$
and~$u=\phi$ in~$D$. Also, for any~$x\in N$, the expected payoff
landing at~$x$ must be equal to the additional payoff~$\psi(x)$
plus the average payoff~$u(y)$ obtained by jumping from~$x$ to~$y\in\Omega$,
that is:
$$ {\mbox{for any $x\in N$, }} \quad
u(x)=\psi(x)+\frac{\displaystyle\int_\Omega\frac{u(y)}{|x-y|^{n+2s}}\,dy}{
\displaystyle\int_\Omega\frac{dy}{|x-y|^{n+2s}}},$$
which corresponds to~$\tilde{\mathcal N}_s u(x)=\psi(x)$.
\end{rem}

\subsection{Weak solutions with Neumann conditions}

The integration by parts formula from Lemma \ref{lema2} leads to the following:

\begin{defi}\label{def weak}
Let $f\in L^{2}(\Omega)$ and~$g\in L^1(\R^n\setminus\Omega)$.
Let $u\in H^s_{\Omega,0}$. We say that $u$ is a weak solution of
\begin{equation}\label{elliptic-pb}
\left\{ \begin{array}{rcll}
(-\Delta)^su&=&f&\textrm{in }\Omega \\
\mathcal N_su&=&g&\textrm{in }\R^n\setminus\overline{\Omega}
\end{array}\right.
\end{equation}
whenever
\begin{equation}\label{weak}
\frac{c_{n,s}}{2}\int_{\R^{2n}\setminus(\mathcal C\Omega)^2}
\frac{\bigl(u(x)-u(y)\bigr)\bigl(v(x)-v(y)\bigr)}{|x-y|^{n+2s}}\,dx\,dy=
\int_\Omega f\,v+\int_{\R^n\setminus\Omega} g\,v
\end{equation}
for all test functions $v\in H^s_{\Omega,g}$.
\end{defi}

With this definition, we can prove the following.

\begin{prop}\label{prop free}
Let $f\in L^2(\Omega)$ and~$g\in L^1(\R^n\setminus\Omega)$.
Let~$I:H^s_{\Omega,g}\rightarrow\R$ be the functional defined as
\[I[u]:=\frac{c_{n,s}}{4}\int_{\R^{2n}\setminus(\mathcal C\Omega)^2}\frac{|u(x)-u(y)|^2}{|x-y|^{n+2s}}\,dx\,dy-
\int_\Omega f\,u-\int_{\R^n\setminus\Omega} g\,u\]
for every~$u\in H^s_{\Omega,g}$.

Then any critical point of~$I$ is a weak solution of \eqref{elliptic-pb}.
\end{prop}

\begin{proof}
First of all, we observe that the functional~$I$ is well defined on~$H^s_{\Omega,g}$.
Indeed, if~$u\in H^s_{\Omega,g}$ then
$$ \left|\int_{\Omega}f\,u\right|\le \|f\|_{L^2(\Omega)}\|u\|_{L^2(\Omega)}\le
C\,\|u\|_{H^s_{\Omega,g}},$$
and
$$ \left|\int_{\R^n\setminus\Omega}g\, u\right|\le
\int_{\R^n\setminus\Omega}|g|^{1/2}\,|g|^{1/2}\, |u|
\le \|g\|_{L^1(\R^n\setminus\Omega)}^{1/2}\, \||g|^{1/2} u\|_{L^2(\R^n\setminus\Omega)}\le
C\,\|u\|_{H^s_{\Omega,g}}. $$
Therefore, if~$u\in H^s_{\Omega,g}$ we have that
$$ |I[u]|\le C\|u\|_{H^s_{\Omega,g}}<+\infty.$$

Now, we compute the first variation of~$I$. For this, we take~$|\epsilon|<1$
and~$v\in H^s_{\Omega,g}$. Then the function~$u+\epsilon v\in H^s_{\Omega,g}$, and so
we can compute
\begin{eqnarray*}
I[u+\epsilon v] &=& \frac{c_{n,s}}{4} \int_{\R^{2n}\setminus(\mathcal C\Omega)^2}
\frac{|(u+\epsilon v)(x)-(u+\epsilon v)(y)|^2}{|x-y|^{n+2s}}\,dx\,dy\\
&&\qquad -
\int_\Omega f(u+\epsilon v)- \int_{\R^n\setminus\Omega} g(u+\epsilon v) \\
&=& I(u) +\epsilon\left(\frac{c_{n,s}}{2}
\int_{\R^{2n}\setminus(\mathcal C\Omega)^2}\frac{(u(x)-u(y))(v(x)-v(y))}{|x-y|^{n+2s}}\,dx\,dy-
\int_\Omega f\,v-\int_{\R^n\setminus\Omega} g\,v\right) \\
&&\qquad + \frac{c_{n,s}}{4}\epsilon^2\int_{\R^{2n}
\setminus(\mathcal C\Omega)^2}\frac{|v(x)-v(y)|^2}{|x-y|^{n+2s}}\,dx\,dy.
\end{eqnarray*}
Hence,
\begin{eqnarray*}
&&\lim_{\epsilon\rightarrow 0}\frac{I[u+\epsilon v]-I[u]}{\epsilon} \\
&&\qquad = \frac{c_{n,s}}{2}
\int_{\R^{2n}\setminus(\mathcal C\Omega)^2}\frac{(u(x)-u(y))(v(x)-v(y))}{|x-y|^{n+2s}}\,dx\,dy-
\int_\Omega f\,v-\int_{\R^n\setminus\Omega} g\,v,
\end{eqnarray*}
which means that
$$ I'[u](v) = \frac{c_{n,s}}{2}
\int_{\R^{2n}\setminus(\mathcal C\Omega)^2}\frac{(u(x)-u(y))(v(x)-v(y))}{|x-y|^{n+2s}}\,dx\,dy-
\int_\Omega f\,v-\int_{\R^n\setminus\Omega} g\,v.$$
Therefore, if $u$ is a critical point of~$I$, then $u$ is a weak solution to~\eqref{elliptic-pb},
according to Definition~\ref{def weak}.
\end{proof}

Next result is a sort of maximum principle and it
is auxiliary towards the existence and uniqueness theory
provided in the subsequent Theorem~\ref{EX}.

\begin{lem}\label{max-prin}
Let~$f\in L^2(\Omega)$ and~$g\in L^1(\R^n\setminus\Omega)$.
Let $u$ be any $H^s_{\Omega,0}$ function satisfying in the weak sense
\[\left\{ \begin{array}{rcll}
(-\Delta)^su&=&f&\textrm{in }\Omega \\
\mathcal N_su&=&g&\textrm{in }\R^n\setminus\overline{\Omega},
\end{array}\right.\]
with $f\geq0$ and $g\geq0$.

Then, $u$ is constant.
\end{lem}

\begin{proof}
First, we observe that the function~$v\equiv 1$ belongs to~$H^s_{\Omega,g}$,
and therefore we can use it as a test function in~\eqref{weak},
obtaining that
\[0\leq \int_\Omega f=-\int_{\R^n\setminus\Omega}g\leq 0.\]
This implies that
\[f=0\quad \textrm{a.e. in}\ \Omega\quad\, \textrm{and}\quad\, g=0\quad \textrm{a.e. in}\ \R^n\setminus\Omega.\]

Therefore, taking $v=u$ as a test function in \eqref{weak}, we deduce that
\[\int_{\R^{2n}\setminus(\mathcal C\Omega)^2}\frac{|u(x)-u(y)|^2}{|x-y|^{n+2s}}\,dx\,dy=0,\]
and hence $u$ must be constant.
\end{proof}

We can now give the following existence and uniqueness result
(we observe that its statement is in complete analogy\footnote{The
only difference with the classical case is that in Theorem~\ref{EX}
it is \emph{not} necessary to suppose that the domain is connected
in order to obtain the uniqueness result.}
with the classical case, see e.g. page~294 in~\cite{Jost}).

\begin{thm}\label{EX}
Let $\Omega\subset\R^n$ be a bounded Lipschitz domain, $f\in L^2(\Omega)$,
and $g\in L^1(\R^n\setminus\Omega)$.
Suppose that there exists a $C^2$ function~$\psi$
such that~$\mathcal{N}_s\psi=g$ in~$\R^n\setminus\overline{\Omega}$.

Then, problem \eqref{elliptic-pb} admits a weak solution in~$H^s_{\Omega,0}$
if and only if
\begin{equation}\label{compatibility}
\int_\Omega f=-\int_{\R^n\setminus\Omega} g.
\end{equation}
Moreover, in case that \eqref{compatibility} holds,
the solution is unique up to an additive constant.
\end{thm}

\begin{proof}
\emph{Case 1}. We do first the case $g\equiv0$, i.e., with homogeneous nonlocal
Neumann conditions. We also assume that~$f\not\equiv 0$, otherwise there is nothing
to prove.

Given~$h\in L^2(\Omega)$, we look for a solution $v\in H^s_{\Omega,g}$ of the problem
\begin{equation}\label{aux}
\int_{\Omega}v\, \varphi +\int_{\R^{2n}\setminus(\mathcal C\Omega)^2}
\frac{(v(x)-v(y))(\varphi(x)-\varphi(y))}{|x-y|^{n+2s}}\,dx\,dy = \int_{\Omega} h\, \varphi,
\end{equation}
for any~$\varphi\in H^s_{\Omega,g}$,
with homogeneous Neumann conditions $\mathcal N_sv=0$ in $\R^n\setminus\overline{\Omega}$.

We consider the functional~$\mathcal{F}:H^s_{\Omega,g}\rightarrow\R$ defined as
$$ \mathcal{F}(\varphi):= \int_{\Omega}h\,\varphi \ {\mbox{ for any }}\varphi\in H^s_{\Omega,g}.$$
It is easy to see that~$\mathcal{F}$ is linear. Moreover, it is continuous on~$H^s_{\Omega,g}$:
$$\left|\mathcal{F}(\varphi)\right|\le \int_{\Omega}|h|\,|\varphi|\le \|h\|_{L^2(\Omega)}\,
\|\varphi\|_{L^2(\Omega)}\le \|h\|_{L^2(\Omega)}\, \|\varphi\|_{H^s_{\Omega,g}}. $$
Therefore, from the Riesz representation theorem it follows that problem~\eqref{aux}
admits a unique solution~$v\in H^s_{\Omega,g}$ for any given~$h\in L^2(\Omega)$.

Furthermore, taking~$\varphi:=v$ in~\eqref{aux}, one obtain that
\begin{equation}\label{aux2}
\|v\|_{H^s(\Omega)}\leq C\|h\|_{L^2(\Omega)}.
\end{equation}

Now, we define the operator $T_o :L^2(\Omega)\longrightarrow H^s_{\Omega,g}$
as $T_o h=v$.
We also define by~$T$ the restriction operator in~$\Omega$, that is
$$ T h = T_o h\big|_\Omega.$$
That is, the function~$T_o h$ is defined in the whole of~$\R^n$,
then we take~$T h$ to be its restriction in~$\Omega$. In this way,
we see that~$T:L^2(\Omega)\longrightarrow
L^2(\Omega)$.

We have that~$T$ is compact. Indeed, we take a sequence~$\{h_k\}_{k\in\N}$
bounded in~$L^2(\Omega)$. Hence, from~\eqref{aux2} we deduce that the sequence
of~$T h_k$ is bounded in~$H^s(\Omega)$, which is compactly embedded in~$L^2(\Omega)$
(see e.g. \cite{DPV}). Therefore, there exists a subsequence that converges in~$L^2(\Omega)$.

Now, we show that~$T$ is self-adjoint. For this, to avoid any smoothness
issue on the test function, we will proceed by approximation.
We take~$h_1, h_2\in C^{\infty}_0(\Omega)$
and we use the weak formulation in~\eqref{aux} to say that, for every~$\varphi, \phi\in H^s_{\Omega,g}$,
we have
\begin{equation}\label{aux3}
\int_{\Omega}T_o h_1\, \varphi +\int_{\R^{2n}\setminus(\mathcal C\Omega)^2}
\frac{(T_o h_1(x)-T_o h_1(y))(\varphi(x)-\varphi(y))}{|x-y|^{n+2s}}\,dx\,dy = \int_{\Omega} h_1\,\varphi,
\end{equation}
and
\begin{equation}\label{aux4}
\int_{\Omega}T_o h_2\, \phi +\int_{\R^{2n}\setminus(\mathcal C\Omega)^2}
\frac{(T_o h_2(x)-T_o h_2(y))(\phi(x)-\phi(y))}{|x-y|^{n+2s}}\,dx\,dy = \int_{\Omega} h_2\, \phi,
\end{equation}
Now we take $\varphi:=T_o h_2$ and~$\phi:=T_o h_1$ in~\eqref{aux3} and~\eqref{aux4} respectively
and we obtain that
$$ \int_{\Omega} h_1\,T_o h_2 = \int_{\Omega} T_o h_1\,h_2$$
for any $h_1,h_2\in C^{\infty}_0(\Omega)$. Accordingly, since~$T_o h_1=Th_1$
and~$T_o h_2=Th_2$ in~$\Omega$, we conclude that
\begin{equation}\label{aux5}
\int_{\Omega} h_1\,Th_2 = \int_{\Omega} Th_1\,h_2
\end{equation}
for any $h_1,h_2\in C^{\infty}_0(\Omega)$.
If~$h_1,h_2\in L^2(\Omega)$, there exist sequences of functions
in~$C^{\infty}_0(\Omega)$, say~$h_{1,k}$ and~$h_{2,k}$, such that
$h_{1,k}\rightarrow h_1$ and~$h_{2,k}\rightarrow h_2$ in~$L^2(\Omega)$ as~$k\rightarrow+\infty$.
From~\eqref{aux5} we have that
\begin{equation}\label{aux6}
\int_{\Omega} h_{1,k}\,Th_{2,k} = \int_{\Omega} Th_{1,k}\,h_{2,k}.
\end{equation}
Moreover, from~\eqref{aux2} we deduce that~$Th_{1,k}\rightarrow Th_1$
and~$Th_{2,k}\rightarrow Th_2$ in~$H^s(\Omega)$ as~$k\rightarrow+\infty$, and so
$$ \int_{\Omega} h_{1,k}\,Th_{2,k}\rightarrow \int_{\Omega} h_1\,Th_2 \ {\mbox{ as }} k\rightarrow+\infty $$
and
$$ \int_{\Omega} Th_{1,k}\,h_{2,k} \rightarrow \int_{\Omega} Th_1\,h_2 \ {\mbox{ as }} k\rightarrow+\infty. $$
The last two formulas and~\eqref{aux6} imply that
\begin{equation}\label{end-aux3} \int_{\Omega} h_1\,Th_2 = \int_{\Omega} Th_1\,h_2 \ {\mbox{ for any }}h_1,h_2\in L^2(\Omega),\end{equation}
which says that~$T$ is self-adjoint.

Now we prove that
\begin{equation}\label{KER}
{\mbox{${Ker}(Id-T)$ consists of constant functions.}}
\end{equation}
First of all, we check that the constants are in~${Ker}(Id-T)$.
We take a function constantly equal to~$c$ and we observe that~$(-\Delta)^s c=0$
in~$\Omega$ (hence~$(-\Delta)^s c+c = c$)
and~${\mathcal{N}}_s c=0$ in~$\R^n\setminus\Omega$. This shows that~$T_o c=c$
in~$\R^n$, and so~$Tc =c$ in~$\Omega$, which implies that~$c\in
{Ker}(Id-T)$.
Viceversa, now we show that if~$v\in{Ker}(Id-T)\subseteq L^2(\Omega)$,
then~$v$ is constant.
For this, we consider~$T_o v\in H^s_{\Omega,g}$. By construction
\begin{equation} \label{Tzer}
(-\Delta)^s (T_ov)+ (T_o v)=v \
{\mbox{ in $\Omega$,}}\end{equation}
in the weak sense, and
\begin{equation}\label{Tz}
{\mbox{${\mathcal{N}}_s (T_o v)=0$ in~$\R^n
\setminus\Omega$.}}\end{equation} On the other hand,
since~$v\in{Ker}(Id-T)$, we have that
\begin{equation}\label{T991}
{\mbox{$v=Tv=T_o v$ in~$\Omega$.}}\end{equation}
Hence, by~\eqref{Tzer}, we have that
$$ (-\Delta)^s (T_ov)=0\
{\mbox{ in $\Omega$.}} $$
Using this, \eqref{Tz} and Lemma \ref{max-prin}, we obtain that~$T_o v$
is constant. Thus, by~\eqref{T991}, we obtain that~$v$ is constant in~$\Omega$
and this completes the proof of~\eqref{KER}.

{F}rom~\eqref{KER} and the Fredholm Alternative, we conclude that
$$ {Im} (Id-T) ={Ker}(Id-T)^\perp
=\{ {\mbox{constant functions}}\}^\perp,$$
where the orthogonality notion is in~$L^2(\Omega)$. More explicitly,
\begin{equation}\label{IM}
{Im} (Id-T)=\left\{ f\in L^2(\Omega)
{\mbox{ s.t. }}\int_\Omega f=0
\right\}.\end{equation}
Now, let us take~$f$ such that~$\int_\Omega f=0$.
By~\eqref{IM}, we know that there exists~$w\in L^2(\Omega)$ such that~$f=w-Tw$.
Let us define $u:=T_o w$. By construction,
we have that~${\mathcal{N}}_s u=0$
in~$\R^n\setminus\Omega$, and that
$$ (-\Delta)^s (T_ow)+(T_ow)=w \ {\mbox{ in }}\Omega.$$
Consequently, in~$\Omega$,
$$ f=w-Tw =w-T_o w = (-\Delta)^s (T_ow) =(-\Delta)^s u,$$
and we found the desired solution in this case.

Viceversa, if we have a solution~$u\in H^s_{\Omega,g}$ of~$(-\Delta)^s u=f$
in~$\Omega$ with~${\mathcal{N}}_s u=0$
in~$\R^n\setminus\Omega$, we set~$w:=f+u$ and we observe that
$$ (-\Delta)^s u+u=f+u=w \ {\mbox{ in }}\Omega.$$
Accordingly, we have that~$u=T_o w$ in~$\R^n$, hence~$u=Tw$ in~$\Omega$.
This says that
$$ (Id-T) w = w -u = f \ {\mbox{ in }}\Omega$$
and so~$f\in {Im} (Id-T)$. Thus, by~\eqref{IM}, we obtain that~$\int_\Omega f=0$.

This establishes the validity of Theorem~\ref{EX} when~$g\equiv0$.
\medskip

\emph{Case 2}. Let us now consider the nonhomogeneous case \eqref{elliptic-pb}.
By the hypotheses, there exists a $C^2$ function~$\psi$ satisfying $\mathcal N_s\psi=g$ in $\R^n\setminus\overline{\Omega}$.

Let $\bar u=u-\psi$.
Then, $\bar u$ solves
\[\left\{ \begin{array}{rcll}
(-\Delta)^s\bar u&=&\bar f&\textrm{in }\Omega \\
\mathcal N_su&=&0&\textrm{in }\R^n\setminus\overline{\Omega},
\end{array}\right.\]
with
\[\bar f=f-(-\Delta)^s\psi.\]
Then, as we already proved, this problem admits a solution if and only if $\int_\Omega \bar f=0$, i.e., if
\begin{equation}\label{E}
0=\int_\Omega \bar f=\int_\Omega f-\int_\Omega (-\Delta)^s \psi.\end{equation}
But, by Lemma~\ref{lema1}, we have that
\[\int_\Omega(-\Delta)^s\psi=-\int_{\R^n\setminus\Omega}
\mathcal N_s\psi=-\int_{\R^n\setminus\Omega}g.\]
{F}rom this and~\eqref{E} we conclude that
a solution exists if and only if \eqref{compatibility} holds.

Finally, the solution is unique up to an additive constant
thanks to Lemma~\ref{max-prin}.
\end{proof}

\subsection{Eigenvalues and eigenfunctions}
Here we discuss the spectral properties of problem \eqref{eq-intro}.
For it, we will need the following classical tool.

\begin{lem}[Poincar\'e inequality]\label{poincare}
Let $\Omega\subset \R^n$ be any bounded Lipschitz domain, and let $s\in(0,1)$.
Then, for all functions $u\in H^s(\Omega)$, we have
\[\int_\Omega \left|u-\ave_\Omega u\right|^2dx\leq C_\Omega\int_{\Omega}\int_\Omega \frac{|u(x)-u(y)|^2}{|x-y|^{n+2s}}\,dx\,dy,\]
where the constant $C_\Omega>0$ depends only on $\Omega$ and $s$.
\end{lem}

\begin{proof} We give the details for the facility of the reader.
We argue by contradiction and we assume that the inequality does not hold.
Then, there exists a sequence of functions $u_k\in H^s(\Omega)$ satisfying
\begin{equation}\label{poinc1}
\ave_\Omega u_k=0,\qquad \|u_k\|_{L^2(\Omega)}=1,
\end{equation}
and
\begin{equation}\label{poinc2}
\int_{\Omega}\int_\Omega \frac{|u_k(x)-u_k(y)|^2}{|x-y|^{n+2s}}\,dx\,dy<\frac1k.
\end{equation}
In particular, the functions $\{u_k\}_{k\geq1}$ are bounded in $H^s(\Omega)$.

Using now that the embedding $H^s(\Omega)\subset L^2(\Omega)$ is compact (see e.g. \cite{DPV}), it follows that a subsequence $\{u_{k_j}\}_{j\geq1}$ converges to a function $\bar u\in L^2(\Omega)$, i.e.,
\[u_{k_j}\longrightarrow \bar u\quad\ \textrm{in}\ L^2(\Omega).\]
Moreover, we deduce from \eqref{poinc1} that
\begin{equation}\label{poinc3}
\ave_\Omega \bar u=0,\qquad\textrm{and}\qquad \|\bar u\|_{L^2(\Omega)}=1.
\end{equation}

On the other hand, \eqref{poinc2} implies that
\[\int_{\Omega}\int_\Omega \frac{|\bar u(x)-\bar u(y)|^2}{|x-y|^{n+2s}}\,dx\,dy=0.\]
Thus, $\bar u$ is constant in $\Omega$, and this contradicts \eqref{poinc3}.
\end{proof}

We finally give the description of the eigenvalues of $(-\Delta)^s$ with zero Neumann boundary conditions.

\begin{thm}\label{EVAL}
Let $\Omega\subset\R^n$ be a bounded Lipschitz domain.
Then, there exist a sequence of nonnegative values
\[0=\lambda_1<\lambda_2\leq \lambda_3\leq\cdots,\] and
a sequence of functions~$u_i:\R^n\to\R$ such that
\[\left\{ \begin{array}{rcll}
(-\Delta)^su_i(x)&=&\lambda_i u_i(x)&\textrm{for any }x\in\Omega \\ \mathcal
N_su_i(x)&=&0&\textrm{for any }x\in\R^n\setminus\overline{\Omega}.
\end{array}\right.\]
Also, the functions~$u_i$ (when restricted to~$\Omega$)
provide a complete orthogonal system in~$L^2(\Omega)$.
\end{thm}

\begin{proof}
We define
$$ L^2_0(\Omega):=\left\{u\in L^2(\Omega) \ :\ \int_{\Omega}u =0\right\}.$$
Let the operator~$T_o$
be defined by $T_of=u$, where $u$ is the
unique solution of
\[\left\{ \begin{array}{rcll}
(-\Delta)^su&=& f&\textrm{in }\Omega \\
\mathcal N_su&=&0&\textrm{in }\R^n\setminus\overline{\Omega}
\end{array}\right.\]
according to Definition~\ref{def weak}.
We remark that the existence and uniqueness of such solution
is a consequence of the fact that~$f\in L^2_0(\Omega)$ and Theorem~\ref{EX}.
We also define~$T$ to be the restriction of~$T_o$ in~$\Omega$, that is
$$ T f =T_o f\Big|_\Omega.$$
In this way~$T:L^2_0(\Omega)\longrightarrow L^2_0(\Omega)$.

Also, we claim that the operator $T$ is compact and self-adjoint.

We first show that~$T$ is compact.
Indeed, taking $v=u=T_o f$ in the weak formulation of the problem \eqref{weak}, we obtain
\begin{equation}\label{F}
\frac{c_{n,s}}{2}\int_{\R^{2n}\setminus (\mathcal C\Omega)^2}\frac{|u(x)-u(y)|^2}{|x-y|^{n+2s}}
\,dx\,dy \leq \|f\|_{L^2(\Omega)}\|u\|_{L^2(\Omega)}.
\end{equation}
Now, using the Poincar\'e inequality in Lemma \ref{poincare} (recall that $\ave_\Omega u=0$),
we deduce that
\begin{equation}\label{L2}
\|u\|_{L^2(\Omega)}\leq C \left(\int_{\Omega\times\Omega}
\frac{|u(x)-u(y)|^2}{|x-y|^{n+2s}}\,dx\,dy\right)^{1/2}.
\end{equation}
This and~\eqref{F} give that
\begin{equation}\label{Hs bound}
\left(\int_{\Omega\times\Omega}
\frac{|u(x)-u(y)|^2}{|x-y|^{n+2s}}\,dx\,dy\right)^{1/2}\leq C\|f\|_{L^2(\Omega)}.
\end{equation}
Now, we take a sequence~$\{f_k\}_{k\in\N}$ bounded in~$L^2(\Omega)$.
{F}rom~\eqref{L2} and~\eqref{Hs bound} we obtain that~$u_k=Tf_k$ is bounded in~$H^s(\Omega)$.
Hence, since the embedding $H^s(\Omega)\subset L^2(\Omega)$ is compact,
there exists a subsequence that converges in~$L^2(\Omega)$.
Therefore, $T$ is compact.

Now we show that~$T$ is self-adjoint in~$L^2_0(\Omega)$.
The proof is very similar to the one in \eqref{aux3}--\eqref{end-aux3},
but for the facility of the reader we give it in full detail
(the reader who is not interested can jump directly to \eqref{NE0}).
To show self-adjointness, we take~$f_1$ and~$f_2$ in~$C^{\infty}_0(\Omega)$,
with~$\ave_\Omega f_1=\ave_\Omega f_2=0$.
Then from the weak formulation in~\eqref{weak} we have that,
for every~$v,w\in H^s_{\Omega,g}$,
\begin{equation}\label{we1}
\frac{c_{n,s}}{2}\int_{\R^{2n}\setminus(\mathcal C\Omega)^2}
\frac{(T_o f_1(x)-T_o f_1(y))(v(x)-v(y))}{|x-y|^{n+2s}}\,dx\,dy = \int_{\Omega}f_1\, v
\end{equation}
and
\begin{equation}\label{we2}
\frac{c_{n,s}}{2}\int_{\R^{2n}\setminus(\mathcal C\Omega)^2}
\frac{(T_o f_2(x)-T_o f_2(y))(w(x)-w(y))}{|x-y|^{n+2s}}\,dx\,dy = \int_{\Omega}f_2\, w.
\end{equation}
We observe that we can take~$v:=T_o f_2$ in~\eqref{we1} and~$w:=T_o f_1$ in~\eqref{we2}
(and recall that~$T_o f_i=Tf_i$ in~$\Omega$),
obtaining that
\begin{equation}\label{we3}
\int_{\Omega}f_1\, Tf_2 = \int_{\Omega}f_2\, Tf_1, \ {\mbox{ for any }}
f_1,f_2\in C^{\infty}_0(\Omega).
\end{equation}
Now, if $f_1, f_2\in L^2_0(\Omega)$ we can find sequences of
functions~$f_{1,k}, f_{2,k}\in C^{\infty}_0(\Omega)$ such that~$f_{1,k}\rightarrow f_1$
and~$f_{2,k}\rightarrow f_2$ in~$L^2(\Omega)$ as~$k\rightarrow+\infty$. Therefore, from~\eqref{we3}, we have
\begin{equation}\label{we4}
\int_{\Omega}f_{1,k}\, Tf_{2,k} = \int_{\Omega}f_{2,k}\, Tf_{1,k}.
\end{equation}
We notice that, thanks to~\eqref{L2} and~\eqref{Hs bound},
$Tf_{1,k}\rightarrow Tf_1$ and $Tf_{2,k}\rightarrow Tf_2$ in~$L^{2}(\Omega)$
as~$k\rightarrow+\infty$,
and therefore, from~\eqref{we4}, we obtain that
$$ \int_{\Omega}f_{1}\, Tf_{2} = \int_{\Omega}f_{2}\, Tf_{1}, $$
thus proving that~$T$ is self-adjoint in $L^2_0(\Omega)$.

Thus, by the spectral theorem there exists a sequence of eigenvalues
$\{\mu_i\}_{i\ge2}$ of~$T$,
and its corresponding eigenfunctions~$\{e_i\}_{i\ge2}$ are a complete orthogonal system in $L^2_0(\Omega)$.

We remark that
\begin{equation}\label{NE0}
\mu_i\ne0.\end{equation}
Indeed, suppose by contradiction that~$\mu_i=0$. Then
\begin{equation}\label{NE1}
0=\mu_ie_i = T e_i =T_o e_i \ {\mbox{ in }}\Omega.\end{equation}
By construction, ${\mathcal{N}}_s (T_o e_i)=0$ in~$\R^n\setminus\overline\Omega$.
This and~\eqref{NE1} give that
$$ T_o e_i (x) =
\frac{\displaystyle\int_\Omega\frac{T_o e_i (y)}{|x-y|^{n+2s}}\,dy}{
\displaystyle\int_\Omega\frac{dy}{|x-y|^{n+2s}}}=0
\ {\mbox{ in }} \R^n\setminus\overline\Omega.$$
Using this and~\eqref{NE1} once again we conclude that~$T_o e_i\equiv0$
a.e. in~$\R^n$. Therefore
$$ 0=(-\Delta)^s (T_o e_i) = e_i \ {\mbox{ in }}\Omega,$$
which gives that~$e_i\equiv0$ in~$\Omega$, hence it is not an eigenfunction.
This establishes~\eqref{NE0}.

{F}rom~\eqref{NE0}, we can define
$$ \lambda_i := \mu_i^{-1}.$$
We also define~$u_i:=T_o e_i$ and we claim that~$u_2, u_3,\cdots$
is the desired system of eigenfunctions,
with corresponding eigenvalues~$\lambda_2,\lambda_3,\cdots$

Indeed,
\begin{equation}\label{tfygdd32uih}
u_i=T_o e_i=Te_i=\mu_i e_i \ {\mbox{ in }}\Omega,\end{equation}
hence the orthogonality and completeness properties of~$u_2, u_3,\cdots$
in~$L^2_0(\Omega)$ follow from those of~$e_2, e_3,\cdots$

Furthermore, in~$\Omega$, we have that~$(-\Delta)^s u_i=
(-\Delta)^s(T_oe_i)=e_i = \lambda_i u_i$,
where~\eqref{tfygdd32uih}
was used in the last step, and this proves the desired
spectral property.

Now, we notice that
\begin{equation}\label{ge0}
{\mbox{$\lambda_i>0$ for any~$i\ge2$.}}\end{equation}
Indeed, its corresponding eigenfunction~$u_i$ solves
\begin{equation}\label{u2}
\left\{ \begin{array}{rcll}
(-\Delta)^su_i &=&\lambda_i u_i &\textrm{in }\Omega \\
\mathcal N_su_i &=&0&\textrm{in }\R^n\setminus\overline{\Omega},
\end{array}\right.
\end{equation}
Then, if we take~$u_i$ as a test function in the weak formulation of~\eqref{u2}, we obtain that
$$ \frac{c_{n,s}}{2}\int_{\R^{2n}\setminus(\mathcal C\Omega)^2}
\frac{|u_i(x)-u_i(y)|^2}{|x-y|^{n+2s}}\,dx\,dy = \lambda_i\int_{\Omega}u_i^2,$$
which implies that~$\lambda_i\ge 0$. Now, suppose by contradiction that~$\lambda_i=0$.
Then, from Lemma~\ref{max-prin} we have that~$u_i$ is constant.
On the other hand, we know that~$u_i\in L^2_0(\Omega)$, and this implies that~$u_i\equiv0$,
which is a contradiction since~$u_i$ is an eigenfunction. This establishes~\eqref{ge0}.

{F}rom~\eqref{ge0}, up to reordering them, we can suppose that~$0<\lambda_2\le\lambda_3\le\cdots$;
now, we notice that $\lambda_1:=0$ is an eigenvalue,
with eigenfunction $u_1:=1$,
thanks to Lemma~\ref{max-prin}. Therefore, we have a sequence of eigenvalues
$0=\lambda_1<\lambda_2\leq \lambda_3\leq\cdots$,
and its corresponding eigenfunctions are a complete orthogonal system in $L^2(\Omega)$.
To check the latter statement, we argue as follows:
first of all, the system~$\{e_i\}_{i\ge1}$
is orthogonal in $L^2(\Omega)$, since we already know
that the system~$\{e_i\}_{i\ge2}$ is orthogonal, and each $e_i$ is orthogonal to $e_1$
for any $i\ge2$,
because $e_i\in L^2_0(\Omega)$ and $e_1\equiv1$.
To check that the system~$\{e_i\}_{i\ge1}$ is complete in $L^2(\Omega)$,
given any $\gamma\in L^2(\Omega)$,
we set
$$ \gamma_1:=\int_\Omega \gamma \; {\mbox{ and }}
\tilde{\gamma}:=\gamma-\gamma_1.$$
Then, $\tilde{\gamma}\in L^2_0(\Omega)$, and so,
since~$\{e_i\}_{i\ge2}$ is a
complete orthogonal system in~$L^2_0(\Omega)$, there exists
a sequence of real numbers~$\{\gamma_i\}_{i\ge2}$ such that
$$ \lim_{N\to+\infty}
\left\| \tilde{\gamma}-\sum_{i=2}^N \gamma_i e_i\right\|_{L^2(\Omega)}=0.$$
Accordingly, since $\tilde{\gamma}=\gamma-\gamma_1 e_1$, we get
$$ \lim_{N\to+\infty}
\left\| \gamma-\sum_{i=1}^N \gamma_i e_i\right\|_{L^2(\Omega)}=0.$$
Since $\gamma$ is an arbitrary function of $L^2(\Omega)$, we have
shown that the system~$\{e_i\}_{i\ge1}$ is complete in $L^2(\Omega)$,
as desired.
This concludes the proof of Theorem~\ref{EVAL}.
\end{proof}

\begin{rem}
We point out that the notion of eigenfunctions in Theorem~\ref{EVAL}
is not completely standard. Indeed, the eigenfunctions~$u_i$
corresponding to the eigenvalues~$\lambda_i$ are defined in the whole of~$\R^n$,
but they satisfy an orthogonality conditions only in~$L^2(\Omega)$.

Alternatively, one can think that the ``natural'' domain of
definition for~$u_i$ is~$\Omega$ itself, since there the eigenvalue
equation~$(-\Delta)^su_i=\lambda_i u_i$ takes place, together
with the orthogonality condition, and then~$u_i$ is ``naturally'' extended
outside~$\Omega$ via the nonlocal Neumann condition. Notice indeed that
the condition~${\mathcal{N}}_s u_i=0$ in~$\R^n\setminus\overline\Omega$
is equivalent to prescribing~$u_i$ outside~$\Omega$ from the values inside~$\Omega$
according to the formula
$$ u_i (x) =
\frac{\displaystyle\int_\Omega\frac{u_i (y)}{|x-y|^{n+2s}}\,dy}{
\displaystyle\int_\Omega\frac{dy}{|x-y|^{n+2s}}}
\ {\mbox{ for any }} x\in\R^n\setminus\overline\Omega.$$
\end{rem}

In the following proposition we deal with the behavior of the solution
of~\eqref{eq-intro} at infinity.

\begin{prop}\label{lim inf}
Let~$\Omega\subset\R^n$ be a bounded domain and let
$u\in H^s_{\Omega,g}$ be a weak solution (according to Definition~\ref{def weak}) of
\[\left\{ \begin{array}{rcll}
(-\Delta)^su&=&f&\textrm{in }\Omega \\
\mathcal N_su&=&0&\textrm{in }\R^n\setminus\overline{\Omega}.
\end{array}\right.\]
Then
\[\lim_{|x|\rightarrow\infty}u(x)=\frac{1}{|\Omega|}\int_\Omega u \ {\mbox{ uniformly in }}x.\]
 \end{prop}

\begin{proof}
First we observe that, since~$\Omega$ is bounded,
there exists~$R>0$ such that~$\Omega\subset B_R$.
Hence, if~$y\in\Omega$, we have that
$$ |x|-R\le |x-y|\le |x|+R, $$
and so
$$ 1-\frac{R}{|x|}\le \frac{|x-y|}{|x|}\le 1+\frac{R}{|x|}.$$
Therefore, given~$\epsilon>0$, there exists~$\bar{R}>R$ such that,
for any~$|x|\ge\bar{R}$, we have
$$ \frac{|x|^{n+2s}}{|x-y|^{n+2s}}=1+\gamma(x,y), $$
where $|\gamma(x,y)|\le\epsilon$.

Recalling the definition of~$\mathcal{N}_su$ given in~\eqref{normal-s}
and using the fact that~$\mathcal{N}_su=0$ in~$\R^n\setminus\overline{\Omega}$, we have that for
any~$x\in\R^n\setminus\overline{\Omega}$
\begin{eqnarray*}
u(x) &=& \frac{\displaystyle\int_{\Omega}\frac{u(y)}{|x-y|^{n+2s}}\,dy}
{\displaystyle\int_{\Omega}\frac{dy}{|x-y|^{n+2s}}}
= \frac{\displaystyle\int_{\Omega}\frac{|x|^{n+2s}u(y)}{|x-y|^{n+2s}}\,dy}
{\displaystyle\int_{\Omega}\frac{|x|^{n+2s}}{|x-y|^{n+2s}}\,dy}\\
&&\qquad = \frac{\displaystyle\int_{\Omega}(1+\gamma(x,y))\,u(y)\,dy}
{\displaystyle\int_{\Omega}(1+\gamma(x,y))\,dy}\\
&&\qquad = \frac{\displaystyle\int_{\Omega}u(y)\,dy +\int_{\Omega}\gamma(x,y)\,u(y)\,dy}
{\displaystyle |\Omega|+\int_{\Omega}\gamma(x,y)\,dy}.
\end{eqnarray*}
We set
\begin{equation*}
\gamma_1(x) :=\ave_{\Omega}\gamma(x,y)\,u(y)\,dy \
{\mbox{ and }} \ \gamma_2(x):=\ave_{\Omega}\gamma(x,y)\,dy,
\end{equation*}
and we notice that~$|\gamma_1(x)|\le C\epsilon$ and~$|\gamma_2(x)|\le\epsilon$, for some~$C>0$.

Hence, we have that for any~$x\in\R^n\setminus\overline{\Omega}$
\begin{eqnarray*}
\left|u(x)-\ave_{\Omega}u(y)\,dy\right| &=&
\left|\frac{\displaystyle\ave_{\Omega}u(y)\,dy +\gamma_1(x)}{1+\gamma_2(x)}
-\ave_{\Omega}u(y)\,dy\right|\\
&=&\frac{\left|\displaystyle \gamma_1(x)-\gamma_2(x)\,\ave_{\Omega}u(y)\,dy\right|}{1+\gamma_2(x)}\\
&\le &\frac{C\,\epsilon}{1-\epsilon}.
\end{eqnarray*}
Therefore, sending~$\epsilon\to 0$ (that is, $|x|\rightarrow +\infty$),
we obtain the desired result.
\end{proof}

\begin{rem}[Interior regularity of solutions]
We notice that, in particular, Proposition~\ref{lim inf} implies that
$u$ is bounded at infinity.
Thus, if solutions are locally bounded,
then one could apply interior regularity results for solutions to $(-\Delta)^su=f$ in $\Omega$
(see e.g. \cite{Landkov,S-obst,CS,SV-weak}).
\end{rem}

\section{The heat equation}\label{S4}

Here we show that solutions of the nonlocal heat equation with zero Neumann datum preserve their mass and have energy that decreases in time.

To avoid technicalities, we assume that $u$ is a classical solution of problem \eqref{eq-heat-intro}, so that we can differentiate under the integral sign.

\begin{prop}\label{mass}
Assume that $u(x,t)$ is a classical solution to \eqref{eq-heat-intro}, in the sense that $u$ is bounded and $|u_t|+|(-\Delta)^su|\leq K$ for all $t>0$.
Then, for all $t>0$,
\[\int_\Omega u(x,t)\,dx=\int_\Omega u_0(x)dx.\]
In other words, the total mass is conserved.
\end{prop}

\begin{proof}
By the dominated convergence theorem, and using Lemma \ref{lema1}, we have
\[\frac{d}{dt}\int_\Omega u=\int_\Omega u_t=-\int_\Omega (-\Delta)^s u=\int_{\R^n\setminus\Omega}\mathcal N_su=0.\]
Thus, the quantity $\int_\Omega u$ does not depend on $t$, and the result follows.
\end{proof}

\begin{prop}\label{eneergy}
Assume that $u(x,t)$ is a classical solution to \eqref{eq-heat-intro}, in the sense that $u$ is bounded and $|u_t|+|(-\Delta)^su|\leq K$ for all $t>0$.
Then, the energy
\[E(t)=\int_{\R^{2n}\setminus(\mathcal C \Omega)^2} \frac{|u(x,t)-u(y,t)|^2}{|x-y|^{n+2s}}\,dx\,dy\]
is decreasing in time $t>0$.
\end{prop}

\begin{proof}
Let us compute $E'(t)$, and we will see that it is negative.
Indeed, using Lemma \ref{lema2},
\[\begin{split}E'(t)&= \frac{d}{dt}\int_{\R^{2n}\setminus(\mathcal C \Omega)^2} \frac{|u(x,t)-u(y,t)|^2}{|x-y|^{n+2s}}\,dx\,dy\\
&=\int_{\R^{2n}\setminus(\mathcal C \Omega)^2} \frac{2\,\bigl(u(x,t)-u(y,t)\bigr)\bigl(u_t(x,t)-u_t(y,t)\bigr)}{|x-y|^{n+2s}}\,dx\,dy\\
&=\frac{4}{c_{n,s}}\int_\Omega u_t\, (-\Delta)^su\,dx,\end{split}\]
where we have used that $\mathcal N_su=0$ in $\R^n\setminus\overline{\Omega}$.

Thus, using now the equation $u_t+(-\Delta)^su=0$ in $\Omega$, we find
\[E'(t)=-\frac{4}{c_{n,s}}\int_\Omega |(-\Delta)^su|^2dx\leq 0,\]
with strict inequality unless $u$ is constant.
\end{proof}

Next we prove that solutions of the nonlocal heat equation
with Neumann condition approach a constant as~$t\to+\infty$:

\begin{prop}\label{89}
Assume that $u(x,t)$ is a classical solution to \eqref{eq-heat-intro}, in the sense that $u$ is bounded and $|u_t|+|(-\Delta)^su|\leq K$ for all $t>0$.
Then,
\[u\,\longrightarrow\,
\frac{1}{|\Omega|}\int_\Omega u_0\quad \textrm{in}\ L^2(\Omega)\]
as $t\rightarrow+\infty$.
\end{prop}

\begin{proof}
Let
\[m:=\frac{1}{|\Omega|}\int_\Omega u_0\]
be the total mass of $u$.
Define also
\[A(t):=\int_\Omega |u-m|^2\,dx.\]
Notice that, by Proposition~\ref{mass}, we have
\[A(t)=\int_\Omega \bigl(u^2-2m u+m^2\bigr)dx=\int_\Omega u^2\,dx-|\Omega|m^2.\]
Then, by Lemma \ref{lema2},
\[A'(t)=2\int_\Omega u_tu\,dx=-2\int_\Omega u(-\Delta)^su\,dx
=-c_{n,s} \int_{\R^{2n}\setminus(\mathcal C \Omega)^2} \frac{|u(x,t)-u(y,t)|^2}{|x-y|^{n+2s}}\,dx\,dy.\]
Hence, $A$ is decreasing.

Moreover, using the Poincar\'e inequality in Lemma \ref{poincare}
and again Proposition~\ref{mass}, we deduce that
\[A'(t)\leq -c\int_\Omega |u-m|^2\,dx=-c\,A(t),\]
for some~$c>0$. Thus, it follows that
\[A(t)\leq e^{-ct} A(0),\]
and thus
\[\lim_{t\rightarrow+\infty}\int_\Omega |u(x,t)-m|^2dx=0,\]
i.e., $u$ converges to $m$ in $L^2(\Omega)$.

Notice that, in fact, we have proved that the convergence is exponentially fast.
\end{proof}

\section{Limits}\label{Slimit}

In this section
we study the limits as $s\rightarrow1$
and the continuity
properties induced by the fractional Neumann condition.

\subsection{Limit as $s\to1$}

\begin{prop}\label{89X}
Let $\Omega\subset\R^n$ be any bounded Lipschitz domain.
Let $u$ and $v$ be $C^2_0(\R^n)$ functions.
Then,
\[\lim_{s\rightarrow 1}\int_{\R^n\setminus\Omega}\mathcal N_s u\,v=\int_{\partial\Omega}\frac{\partial u}{\partial\nu}\,v.\]
\end{prop}

\begin{proof}
By Lemma \ref{lema2}, we have that
\begin{equation}\label{bb40}
\int_{\R^n\setminus\Omega}\mathcal N_s u\,v=\frac{c_{n,s}}{2}\int_{\R^{2n}\setminus (\mathcal{C}\Omega)^2}\frac{(u(x)-u(y))(v(x)-v(y))}{|x-y|^{n+2s}}\,dx\,dy
-\int_\Omega v(-\Delta)^su.
\end{equation}

Now, we claim that
\begin{equation}\label{bb70}
\lim_{s\to 1}\frac{c_{n,s}}{2}\int_{\R^{2n}\setminus (\mathcal{C}\Omega)^2}\frac{(u(x)-u(y))(v(x)-v(y))}{|x-y|^{n+2s}}\,dx\,dy
= \int_{\Omega}\nabla u\cdot\nabla v.
\end{equation}
We observe that to show~\eqref{bb70}, it is enough to prove that, for any~$u\in C^2_0(\R^n)$,
\begin{equation}\label{bb71}
\lim_{s\to 1}\frac{c_{n,s}}{2}\int_{\R^{2n}\setminus (\mathcal{C}\Omega)^2}\frac{|u(x)-u(y)|^2}{|x-y|^{n+2s}}\,dx\,dy
= \int_{\Omega}|\nabla u|^2.
\end{equation}
Indeed,
\begin{eqnarray*}
&& \int_{\R^{2n}\setminus (\mathcal{C}\Omega)^2}\frac{(u(x)-u(y))(v(x)-v(y))}{|x-y|^{n+2s}}\,dx\,dy \\
&&\qquad\qquad = \frac12 \int_{\R^{2n}\setminus (\mathcal{C}\Omega)^2}\frac{|(u+v)(x)-(u+v)(y)|^2}{|x-y|^{n+2s}}\,dx\,dy \\
&&\qquad\qquad\qquad -\frac12 \int_{\R^{2n}\setminus (\mathcal{C}\Omega)^2}\frac{|u(x)-u(y)|^2}{|x-y|^{n+2s}}\,dx\,dy\\
&&\qquad\qquad\qquad -\frac12 \int_{\R^{2n}\setminus (\mathcal{C}\Omega)^2}\frac{|v(x)-v(y)|^2}{|x-y|^{n+2s}}\,dx\,dy.
\end{eqnarray*}

Now, we recall that
$$ \lim_{s\to 1}\frac{c_{n,s}}{1-s}= \frac{4n}{\omega_{n-1}}, $$
(see Corollary~4.2 in~\cite{DPV}), and so we have to show that
\begin{equation}\label{bb500}
\lim_{s\to 1}\,(1-s)\int_{\R^{2n}\setminus (\mathcal{C}\Omega)^2}
\frac{|u(x)-u(y)|^2}{|x-y|^{n+2s}}\, dx\, dy =
\frac{\omega_{n-1}}{2n} \int_{\Omega}|\nabla u|^2.
\end{equation}
For this, we first show that
\begin{equation}\label{CL00}
\lim_{s\to 1}\,(1-s)\int_{\Omega\times(\mathcal{C}\Omega)}
\frac{|u(x)-u(y)|^2}{|x-y|^{n+2s}}\, dx\, dy =0.
\end{equation}
Without loss of generality, we can suppose that~$B_r\subset\Omega\subset B_R$,
for some~$0<r<R$. Since~$u\in C^2_0(\R^n)$, then
\begin{eqnarray*}
\int_{\Omega\times(\mathcal{C}\Omega)}
\frac{|u(x)-u(y)|^2}{|x-y|^{n+2s}}\, dx\, dy
&\le & 4\|u\|^2_{L^\infty(\R^n)} \int_{\Omega\times(\mathcal{C}\Omega)}
\frac{1}{|x-y|^{n+2s}}\, dx\, dy \\
&\le & 4\|u\|^2_{L^\infty(\R^n)} \int_{B_R\times(\mathcal{C}B_r)}
\frac{1}{|x-y|^{n+2s}}\, dx\, dy \\
&\le & 4\|u\|^2_{L^\infty(\R^n)}\,\omega_{n-1}\int_{B_R}\,dx\,\int_r^{+\infty}\rho^{n-1}\rho^{-n-2s}\,d\rho \\
&=& 4\|u\|^2_{L^\infty(\R^n)}\,\omega_{n-1}\int_{B_R}\,dx\,\int_r^{+\infty}\rho^{-1-2s}\,d\rho \\
&=& 4\|u\|^2_{L^\infty(\R^n)}\,\frac{\omega_{n-1}\,r^{-2s}}{2s}\int_{B_R}\,dx\\
&=& 4\|u\|^2_{L^\infty(\R^n)}\,\frac{\omega_{n-1}^2\,R^n\,r^{-2s}}{2s},
\end{eqnarray*}
which implies~\eqref{CL00}.
Hence,
\begin{equation}\begin{split}\label{constant}
&\lim_{s\to 1}\,(1-s)\int_{\R^{2n}\setminus (\mathcal{C}\Omega)^2}
\frac{|u(x)-u(y)|^2}{|x-y|^{n+2s}}\, dx\, dy \\
&\qquad =\lim_{s\to 1}\,(1-s)\int_{\Omega\times\Omega}
\frac{|u(x)-u(y)|^2}{|x-y|^{n+2s}}\, dx\, dy = C_{n}\,\int_{\Omega}|\nabla u|^2,
\end{split}\end{equation}
where~$C_n>0$ depends only on the dimension, see~\cite{BBM}.

In order to determine the constant~$C_n$, we take a~$C^2$-function~$u$
supported in~$\Omega$. In this case, we have
\begin{equation}\label{rpeoohjb}
\int_{\Omega}|\nabla u|^2\,dx = \int_{\R^n}|\nabla u|^2\,dx
= \int_{\R^n}|\xi|^2\,|\hat{u}(\xi)|^2\,d\xi,
\end{equation}
where~$\hat{u}$ is the Fourier transform of~$u$.
Moreover,
\begin{eqnarray*}
\int_{\R^{2n}\setminus(\mathcal{C}\Omega)^2}\frac{|u(x)-u(y)|^2}{|x-y|^{n+2s}}\,dx\,dy
&=& \int_{\R^{2n}}\frac{|u(x)-u(y)|^2}{|x-y|^{n+2s}}\,dx\,dy \\
&=& 2\,c_{n,s}^{-1}\int_{\R^n}|\xi|^{2s}\,|\hat{u}(\xi)|^2\,dx,
\end{eqnarray*}
thanks to Proposition~3.4 in~\cite{DPV}. Therefore, using Corollary~4.2 in~\cite{DPV}
and~\eqref{rpeoohjb}, we have
\begin{eqnarray*}
&&\lim_{s\to 1}\,(1-s)\int_{\R^{2n}\setminus (\mathcal{C}\Omega)^2}
\frac{|u(x)-u(y)|^2}{|x-y|^{n+2s}}\, dx\, dy \\
&&\qquad =\lim_{s\to 1}\frac{2(1-s)}{c_{n,s}}\int_{\R^n}|\xi|^{2s}\,|\hat{u}(\xi)|^2\,dx\\
&&\qquad = \frac{\omega_{n-1}}{2n}\,\int_{\R^n}|\xi|^2\,|\hat{u}(\xi)|^2\,dx\\
&&\qquad = \frac{\omega_{n-1}}{2n}\,\int_{\Omega}|\nabla u|^2\,dx.
\end{eqnarray*}
Hence, the constant in~\eqref{constant} is~$C_n=\frac{\omega_{n-1}}{2n}$.
This concludes the proof of~\eqref{bb500}, and in turn of~\eqref{bb70}.

On the other hand,
\[-(-\Delta)^su\rightarrow \Delta u\qquad \textrm{uniformly in}\ \R^n,\]
(see Proposition~4.4 in~\cite{DPV}).
This, \eqref{bb40} and \eqref{bb70} give
\[\lim_{s\rightarrow 1}\int_{\R^n\setminus\Omega}\mathcal N_s u\,v=\int_\Omega \nabla u\cdot\nabla v+\int_\Omega v\,\Delta u=\int_{\partial\Omega}\frac{\partial u}{\partial\nu}\,v,\]
as desired.
\end{proof}

\subsection{Continuity properties}

Following is a continuity result for functions satisfying
the nonlocal Neumann condition:

\begin{prop}\label{cont}
Let~$\Omega\subset\R^n$ be a domain with~$C^1$ boundary.
Let~$u$ be continuous in~$\overline\Omega$, with~${\mathcal{N}}_s u=0$
in~$\R^n\setminus\overline\Omega$. Then~$u$ is continuous
in the whole of~$\R^n$.
\end{prop}

\begin{proof} First, let us fix~$x_0\in \R^n\setminus\overline\Omega$.
Since the latter is an open set,
there exists~$\rho>0$ such that~$|x_0-y|
\ge\rho$ for any~$y\in\Omega$. Thus, if~$x\in B_{\rho/2}(x_0)$,
we have that~$|x-y|\ge |x_0-y|-|x_0-x|\ge\rho/2$.

Moreover, if~$x\in B_{\rho/2}(x_0)$, we have that
$$|x-y|\ge |y| -|x_0| -|x_0-x| \ge \frac{|y|}{2}
+\left(\frac{|y|}{4} -|x_0| \right) +\left(\frac{|y|}{4} -\frac\rho2\right)
\ge \frac{|y|}{2},$$
provided that~$|y|\ge R := 4|x_0|+2\rho$.
As a consequence, for any~$x\in B_{\rho/2}(x_0)$, we have that
$$ \frac{|u(y)|+1}{|x-y|^{n+2s}} \le 2^{n+2s}\,
(\|u\|_{L^\infty(\overline\Omega)}+1)\,
\left( \frac{\chi_{B_{R}}(y)}{\rho^{n+2s}}+
\frac{\chi_{\R^n\setminus B_{R}}(y)}{|y|^{n+2s}}\right)=:\psi(y)$$
and the function~$\psi$ belongs to~$L^1(\R^n)$. Thus,
by the Neumann condition and the Dominated Convergence Theorem, we obtain that
$$ \lim_{x\to x_0} u(x)
=\lim_{x\to x_0} \frac{
\displaystyle\int_\Omega \frac{u(y)}{|x-y|^{n+2s}}\,dy}{
\displaystyle\int_\Omega \frac{dy}{|x-y|^{n+2s}}}
=\frac{
\displaystyle\int_\Omega \frac{u(y)}{|x_0-y|^{n+2s}}\,dy}{
\displaystyle\int_\Omega \frac{dy}{|x_0-y|^{n+2s}}}
=u(x_0).$$
This proves that~$u$ is continuous
at any points of~$\R^n\setminus\overline\Omega$.

Now we show the continuity at a point~$p\in\partial\Omega$.
We take a sequence~$p_k\to p$ as~$k\to+\infty$.
We let~$q_k$ be the projection of~$p_k$ to~$\overline\Omega$.
Since~$p\in\overline\Omega$, we have from the minimizing
property of the projection that
$$|p_k-q_k|=\inf_{\xi\in\overline\Omega} |p_k-\xi|\le |p_k-p|,$$
and so
$$ |q_k-p|\le |q_k-p_k|+|p_k-p|\le 2|p_k-p|\to 0$$
as $k\to+\infty$. Therefore,
since we already know from
the assumptions the continuity of~$u$ at~$\overline\Omega$,
we obtain that
\begin{equation}\label{long}
\lim_{k\to+\infty} u(q_k)=u(p).\end{equation}
Now we claim that
\begin{equation}\label{long2}
\lim_{k\to+\infty} u(p_k)-u(q_k)=0.\end{equation}
To prove it, it is enough to consider the points
of the sequence~$p_k$ that belong to~$\R^n\setminus\overline\Omega$
(since, of course, the points~$p_k$ belonging to~$\overline\Omega$ satisfy~$p_k=q_k$
and for them~\eqref{long2} is obvious).
We define~$\nu_k:=(p_k-q_k)/|p_k-q_k|$. Notice that~$\nu_k$
is the exterior normal of~$\Omega$ at~$q_k\in\partial\Omega$.
We consider a rigid motion~${\mathcal{R}}_k$
such that~${\mathcal{R}}_k q_k =0$ and~${\mathcal{R}}_k \nu_k =
e_n=(0,\cdots,0,1)$.
Let also~$h_k:=|p_k-q_k|$.
Notice that
\begin{equation}\label{90X} h_k^{-1} {\mathcal{R}}_k p_k
=h_k^{-1} {\mathcal{R}}_k (p_k-q_k)={\mathcal{R}}_k \nu_k=e_n.\end{equation}
Then, the domain
$$ \Omega_k:= h_k^{-1}{\mathcal{R}}_k \Omega $$
has vertical exterior normal at~$0$ and approaches the halfspace~$\Pi:=
\{x_n<0\}$ as~$k\to+\infty$.

Now, we use the Neumann
condition at~$p_k$ and we obtain that
\begin{eqnarray*} && u(p_k)-u(q_k)=
\frac{
\displaystyle\int_\Omega \frac{u(y)}{|p_k-y|^{n+2s}}\,dy}{
\displaystyle\int_\Omega \frac{dy}{|p_k-y|^{n+2s}}}
-u(q_k)\\ &&\qquad =
\frac{
\displaystyle\int_\Omega \frac{u(y)-u(q_k)}{|p_k-y|^{n+2s}}\,dy}{
\displaystyle\int_\Omega \frac{dy}{|p_k-y|^{n+2s}}} = I_1+I_2,\end{eqnarray*}
with
\begin{eqnarray*}
&& I_1:=\frac{
\displaystyle\int_{\Omega\cap B_{\sqrt{h_k}}(q_k)} \frac{u(y)-u(q_k)}{
|p_k-y|^{n+2s}}\,dy}{
\displaystyle\int_\Omega \frac{dy}{|p_k-y|^{n+2s}}} \\
{\mbox{and }}
&& I_2:=\frac{
\displaystyle\int_{\Omega\setminus B_{\sqrt{h_k}}(q_k)} \frac{u(y)-u(q_k)}{
|p_k-y|^{n+2s}}\,dy}{
\displaystyle\int_\Omega \frac{dy}{|p_k-y|^{n+2s}}}.\end{eqnarray*}
We observe that
the uniform continuity of~$u$ in~$\overline\Omega$
gives that
$$ \lim_{k\to+\infty} \sup_{y\in \Omega\cap B_{\sqrt{h_k}}(q_k)} |u(y)-u(q_k)|
=0.$$
As a consequence
\begin{equation}\label{90I1}
|I_1|\le \sup_{y\in \Omega\cap B_{\sqrt{h_k}}(q_k)} |u(y)-u(q_k)| \to0\end{equation}
as $k\to+\infty$.
Moreover, exploiting
the change of variable~$\eta:=
h_k^{-1}{\mathcal{R}}_k y$ and recalling~\eqref{90X}, we obtain that
\begin{eqnarray*}
|I_2|&\le&\frac{
\displaystyle\int_{\Omega\setminus B_{\sqrt{h_k}}(q_k)} \frac{|u(y)-u(q_k)|}{
|p_k-y|^{n+2s}}\,dy}{
\displaystyle\int_\Omega \frac{dy}{|p_k-y|^{n+2s}}} \\
&\le& 2\|u\|_{L^\infty(\overline\Omega)}
\frac{
\displaystyle\int_{\Omega\setminus B_{\sqrt{h_k}}(q_k)} \frac{dy}{
|p_k-y|^{n+2s}}}{
\displaystyle\int_\Omega \frac{dy}{|p_k-y|^{n+2s}}} \\
&=& 2\|u\|_{L^\infty(\overline\Omega)}
\frac{
\displaystyle\int_{\Omega_k\setminus B_{1/\sqrt{h_k}}} \frac{d\eta}{
|e_n-\eta|^{n+2s}}}{
\displaystyle\int_{\Omega_k} \frac{d\eta}{|e_n-\eta|^{n+2s}}}.
\end{eqnarray*}
Notice that, if~$\eta\in \Omega_k\setminus B_{1/\sqrt{h_k}}$
then
\begin{eqnarray*} && |e_n-\eta|^{n+2s} = |e_n-\eta|^{n+s} |e_n-\eta|^{s}
\ge |e_n-\eta|^{n+s} \Big( |\eta|-1\Big)^s \\ &&\qquad\ge
|e_n-\eta|^{n+s} \Big( h_k^{-1/2}-1\Big)^s \ge
|e_n-\eta|^{n+s} h_k^{-s/4}\end{eqnarray*}
for large~$k$.
Therefore
$$ |I_2|\le
2h_k^{s/4} \|u\|_{L^\infty(\overline\Omega)}
\frac{
\displaystyle\int_{\Omega_k} \frac{d\eta}{
|e_n-\eta|^{n+s}}}{
\displaystyle\int_{\Omega_k} \frac{d\eta}{|e_n-\eta|^{n+2s}}}.$$
Since
$$ \lim_{k\to+\infty}\frac{
\displaystyle\int_{\Omega_k} \frac{d\eta}{
|e_n-\eta|^{n+s}}}{
\displaystyle\int_{\Omega_k} \frac{d\eta}{|e_n-\eta|^{n+2s}}}=
\frac{
\displaystyle\int_{\Pi} \frac{d\eta}{
|e_n-\eta|^{n+s}}}{
\displaystyle\int_{\Pi} \frac{d\eta}{|e_n-\eta|^{n+2s}}},$$
we conlude that~$|I_2|\to0$ as~$k\to+\infty$. This and~\eqref{90I1}
imply~\eqref{long2}.

{F}rom \eqref{long} and~\eqref{long2}, we conclude that
$$ \lim_{k\to+\infty} u(p_k)=u(p),$$
hence~$u$ is continuous at~$p$.
\end{proof}

As a direct consequence of Proposition~\ref{cont}
we obtain:

\begin{cor}\label{C con}
Let~$\Omega\subset\R^n$ be a domain with~$C^1$ boundary.
Let~$v_0\in C(\R^n)$.
Let
$$ v(x):=\left\{
\begin{matrix}
v_0(x) & {\mbox{ if $x\in\overline\Omega$,}}\\
\\
\frac{\displaystyle\int_\Omega
\displaystyle\frac{v_0(y)}{|x-y|^{n+2s}}\,dy}{
\displaystyle\int_\Omega \displaystyle\frac{dy}{|x-y|^{n+2s}}} & {\mbox{ if
$x\in\R^n\setminus\overline\Omega$.}}
\end{matrix}\right. $$
Then $v\in C(\R^n)$ and it satisfies~$v=v_0$ in~$\overline\Omega$
and~${\mathcal{N}}_s v=0$ in~$\R^n\setminus\overline\Omega$.
\end{cor}

\begin{proof} By construction, $v=v_0$ in~$\overline\Omega$
and~${\mathcal{N}}_s v=0$ in~$\R^n\setminus\overline\Omega$.
Then we can use Proposition~\ref{cont} and obtain that~$v\in C(\R^n)$.
\end{proof}

Now we study the boundary behavior of the nonlocal
Neumann function~$ \tilde{\mathcal N}_s u$.

\begin{prop}\label{89Y}
Let $\Omega\subset\R^n$ be a $C^1$ domain, and $u\in C(\R^n)$.
Then, for all $s\in(0,1)$,
\begin{equation}\label{R51}
\lim_{{x\rightarrow\partial\Omega}\atop{x\in\R^n\setminus\overline\Omega}}\tilde{\mathcal N}_s u(x)=0.\end{equation}
Also, if $s>\frac12$ and~$u\in C^{1,\alpha}(\R^n)$ for some~$\alpha\in(0,\,2s-1)$, then
\begin{equation}\label{R52}\partial_\nu\tilde{\mathcal N}_s u(x):=
\lim_{\epsilon\to0^+} \frac{\tilde{\mathcal N}_s u(x+\epsilon\nu)
}{\epsilon}=\kappa \, \partial_\nu u\qquad \textrm{for any}\ x\in\partial\Omega,
\end{equation}
for some constant $\kappa>0$.
\end{prop}

\begin{proof}
Let~$x_k$ be a sequence in~$\R^n\setminus\overline\Omega$
such that~$x_k\to x_\infty\in\partial\Omega$ as~$k\to+\infty$.

By Corollary~\ref{C con} (applied here with~$v_0:=u$),
there exists~$v\in C(\R^n)$ such that~$v=u$
in~$\overline\Omega$
and~${\mathcal{N}}_s v=0$ in~$\R^n\setminus\overline\Omega$.
By the continuity of~$u$ and~$v$ we have that
\begin{equation}\label{89.89}
\lim_{k\to+\infty} u(x_k)-v(x_k)=u(x_\infty)-v(x_\infty)=0.\end{equation}
Moreover
\begin{eqnarray*}
\tilde{\mathcal{N}}_s u(x_k) &=&
\tilde{\mathcal{N}}_s u(x_k) - \tilde{\mathcal{N}}_s v(x_k) \\
&=& \frac{
\displaystyle\int_\Omega \displaystyle\frac{ u(x_k)-u(y) }{ |x_k-y|^{n+2s} }\,dy
- \displaystyle\int_\Omega \displaystyle\frac{v(x_k)-v(y)}{|x_k-y|^{n+2s}}\,dy}{
\displaystyle\int_\Omega \displaystyle\frac{dy}{|x_k-y|^{n+2s}}
} \\
&=&  \frac{ \displaystyle\int_\Omega \displaystyle\frac{u(x_k)-v(x_k)}{|x_k-y|^{n+2s}}\,dy}{
\displaystyle\int_\Omega \displaystyle\frac{dy}{|x_k-y|^{n+2s}} }\\
&=& u(x_k)-v(x_k).
\end{eqnarray*}
This and~\eqref{89.89} imply that
$$ \lim_{k\to+\infty}\tilde{\mathcal{N}}_s u(x_k)=0,$$
that is \eqref{R51}.

Now, we prove~\eqref{R52}.
For this, we suppose that~$s>\frac12$, that~$0\in\partial\Omega$
and that the exterior normal~$\nu$ coincides with~$e_n=(0,\cdots,0,1)$;
then we use~\eqref{R51} and the change
of variable~$\eta:=\epsilon^{-1} y$ in the following computation:
\begin{eqnarray*}
\epsilon^{-1} \Big(\tilde{\mathcal{N}}_s u(\epsilon e_n) -
\tilde{\mathcal{N}}_s u(0)\Big) &=&
\epsilon^{-1} \tilde{\mathcal{N}}_s u(\epsilon e_n) \\&=&
\frac{ \epsilon^{-1}
\displaystyle\int_{\Omega} \displaystyle\frac{u(\epsilon e_n)-u(y)}{
|\epsilon e_n-y|^{n+2s}}\,dy}{
\displaystyle\int_{\Omega} \displaystyle\frac{dy}{|\epsilon e_n-y|^{n+2s}}
}\\ &=&
\frac{ \epsilon^{-1}
\displaystyle\int_{\frac1\epsilon \Omega} \displaystyle\frac{u(\epsilon e_n)-
u(\epsilon\eta)}{
|e_n-\eta|^{n+2s}}\,d\eta
}{
\displaystyle\int_{\frac1\epsilon \Omega} \displaystyle\frac{d\eta}{|e_n-\eta|^{n+2s}}
}=
I_1+I_2,
\end{eqnarray*}
where
\begin{eqnarray*}
&& I_1:=
\frac{
\displaystyle\int_{\frac1\epsilon \Omega} \displaystyle\frac{\nabla u(\epsilon e_n)\cdot
(e_n-\eta) }{
|e_n-\eta|^{n+2s}}\,d\eta
}{
\displaystyle\int_{\frac1\epsilon \Omega} \displaystyle\frac{d\eta}{|e_n-\eta|^{n+2s}} }
\\ {\mbox{and }}&& I_2:=
\frac{ \epsilon^{-1}
\displaystyle\int_{\frac1\epsilon \Omega} \displaystyle\frac{u(\epsilon e_n)-
u(\epsilon\eta)-\epsilon\nabla u(\epsilon e_n)\cdot(e_n-\eta)}{
|e_n-\eta|^{n+2s}}\,d\eta
}{
\displaystyle\int_{\frac1\epsilon \Omega} \displaystyle\frac{d\eta}{|e_n-\eta|^{n+2s}}
}.\end{eqnarray*}
So, if~$\Pi:=\{ x_n<0\}$,
we have that
\begin{eqnarray*}
\lim_{\epsilon\to0^+} I_1 &=&
\frac{
\displaystyle\int_{\Pi} \displaystyle\frac{\nabla u(0)\cdot
(e_n-\eta) }{
|e_n-\eta|^{n+2s}}\,d\eta
}{
\displaystyle\int_{\Pi} \displaystyle\frac{d\eta}{|e_n-\eta|^{n+2s}} }
\\ &=&
\frac{
\displaystyle\int_{\Pi} \displaystyle\frac{\partial_n u(0)
(1-\eta_n) }{
|e_n-\eta|^{n+2s}}\,d\eta
}{
\displaystyle\int_{\Pi} \displaystyle\frac{d\eta}{|e_n-\eta|^{n+2s}} },
\end{eqnarray*}
where we have used that, for any~$i\in\{1,\cdots,n-1\}$
the map~$\eta\mapsto \frac{\partial_i u(0)\cdot
\eta_i }{
|e_n-\eta|^{n+2s}}$ is odd and so its integral averages to zero.
So, we can write
\begin{equation}\label{768}
\lim_{\epsilon\to0^+} I_1 =\kappa\, \partial_n u(0) \ {\mbox{ with }} \
\kappa:=
\frac{
\displaystyle\int_{\Pi} \displaystyle\frac{
(1-\eta_n) }{
|e_n-\eta|^{n+2s}}\,d\eta
}{
\displaystyle\int_{\Pi} \displaystyle\frac{d\eta}{|e_n-\eta|^{n+2s}} }.
\end{equation}
We remark that~$\kappa$ is finite, since~$s>\frac12$.
Moreover
\begin{eqnarray*}
&& \epsilon^{-1}\,\Big|u(\epsilon e_n)-
u(\epsilon\eta)-\epsilon\nabla u(\epsilon e_n)\cdot(e_n-\eta)\Big|\\
&=& \left| \int_0^1 \Big(\nabla u(t\epsilon e_n +(1-t)\epsilon\eta)
-\nabla u(\epsilon e_n)\Big)\cdot(e_n-\eta)\,dt
\right| \\&\le& \|u\|_{C^{1,\alpha(\R^n)}} \,|e_n-\eta|\,
\int_0^1 |t\epsilon e_n +(1-t)\epsilon\eta-\epsilon e_n|^\alpha\,dt
\\&\le& \|u\|_{C^{1,\alpha(\R^n)}} \epsilon^\alpha\,|e_n-\eta|^{1+\alpha}
.\end{eqnarray*}
As a consequence
$$ \epsilon^{-\alpha} |I_2|\le
\frac{\|u\|_{C^{1,\alpha(\R^n)}}
\displaystyle\int_{\frac1\epsilon \Omega} \displaystyle\frac{d\eta}{
|e_n-\eta|^{n+2s-1-\alpha}}
}{
\displaystyle\int_{\frac1\epsilon \Omega}
\displaystyle\frac{d\eta}{|e_n-\eta|^{n+2s}} }
\ \longrightarrow \
\frac{ \|u\|_{C^{1,\alpha(\R^n)}}
\displaystyle\int_{\Pi} \displaystyle\frac{d\eta}{
|e_n-\eta|^{n+2s-1-\alpha}}
}{
\displaystyle\int_{\Pi} \displaystyle\frac{d\eta}{|e_n-\eta|^{n+2s}} }$$
as~$\epsilon\to 0$,
which is finite, thanks to our assumptions on~$\alpha$. This shows that~$
I_2\to0$ as~$\epsilon\to 0$. Hence, recalling~\eqref{768}, we get that
$$\lim_{\epsilon\to0^+} \epsilon^{-1} \Big(
\tilde{\mathcal{N}}_s u(\epsilon e_n) -
\tilde{\mathcal{N}}_s u(0)\Big)=\kappa\, \partial_n u(0),$$
which establishes~\eqref{R52}.
\end{proof}

\section{An overdetermined problem}\label{sover}

In this section we consider an overdetemined problem.
For this, we will use the renormalized nonlocal Neumann condition
that has been introduced in Remark~\ref{rem1}.
Indeed, as we pointed out in Remark~\ref{rem2}, this is
natural if one considers nonhomogeneous Neumann conditions.

\begin{thm}
Let~$\Omega\subset\R^n$ be a bounded and Lipschitz domain.
Then there exists no function~$u\in C(\R^n)$ satisfying
\begin{equation}\label{over}
\left\{ \begin{array}{rcll}
u(x) &=&0 &\textrm{for any }x\in\R^n\setminus\Omega \\
\tilde{\mathcal N}_su(x)&=&1 &\textrm{for any }x\in\R^n\setminus\overline{\Omega}.
\end{array}\right.
\end{equation}
\end{thm}

\begin{rem}
We notice that~$u=\chi_{\Omega}$ satisfies~\eqref{over},
but it is a discontinuous function.
\end{rem}

\begin{proof}
Without loss of generality, we can suppose that~$0\in\partial\Omega$.
We argue by contradiction and we assume that there exists a continuous function~$u$
that satisfies~\eqref{over}. Therefore, there exists~$\delta>0$ such that
\begin{equation}\label{contin}
{\mbox{$|u|\le 1/2$ in~$B_\delta$.}}
\end{equation}

Since~$\Omega$ is Lipschitz, up to choosing~$\delta$ small enough,
we have that~$\Omega\cap B_\delta=\tilde{\Omega}\cap B_\delta$, where
$$ \tilde{\Omega}:=\{ x=(x',x_n)\in\R^{n-1}\times\R {\mbox{ s.t. }} x_n<\gamma(x')\}$$
for a suitable Lipschitz function~$\gamma:\R^{n-1}\rightarrow\R$
such that~$\gamma(0)=0$ and~$\partial_{x'}\gamma(0)=0$.

Now we let~$x:=\epsilon\, e_n\in\R^n\setminus\overline{\Omega}$, for suitable~$\epsilon>0$
sufficiently small. We observe that
\begin{equation}\label{zero}
u(\epsilon\, e_n)=0.
\end{equation}
Moreover we consider the set
$$ \frac{1}{\epsilon}\tilde{\Omega} =
\left\{ y=(y',y_n)\in\R^{n-1}\times\R {\mbox{ s.t. }} y_n<\frac{1}{\epsilon}\gamma(\epsilon y')\right\}.
$$
We also define
$$ K:= \left\{ y=(y',y_n)\in\R^{n-1}\times\R {\mbox{ s.t. }} y_n<-L\,|y'|\right\},
$$
where~$L$ is the Lipschitz constant of~$\gamma$.

We claim that
\begin{equation}\label{kappa}
K\subseteq\epsilon^{-1}\,\tilde{\Omega}.
\end{equation}
Indeed, since~$\gamma$ is Lipschitz and~$0\in\partial\Omega$, we have that
$$ -\gamma(\epsilon y')=-\gamma(\epsilon y')+\gamma(0)\le L\, \epsilon\, |y'|, $$
and so, if~$y\in K$,
$$ y_n\le -L\,|y'|\le \frac{1}{\epsilon}\gamma(\epsilon y'),$$
which implies that~$y\in\epsilon^{-1}\tilde{\Omega}$. This shows~\eqref{kappa}.

Now we define
$$ \Sigma_{\epsilon}:=\int_{B_\delta\cap\Omega}\frac{dy}{|\epsilon e_n-y|^{n+2s}}, $$
and we observe that
\begin{equation}\label{conto1}
\int_{B_\delta\cap\Omega}\frac{u(y)-u(\epsilon\, e_n)}{|\epsilon e_n-y|^{n+2s}}\, dy\le
\frac{1}{2}\, \Sigma_\epsilon, \end{equation}
thanks to~\eqref{zero} and~\eqref{contin}.
Furthermore, if~$y\in\R^n\setminus B_\delta$ and~$\epsilon\le\delta/2$, we have
$$ |y-\epsilon e_n|\ge |y|-\epsilon\ge \frac{|y|}{2}, $$
which implies that
\begin{equation}\label{conto2}
\int_{\Omega\setminus B_\delta}\frac{u(y)-u(\epsilon\, e_n)}{|\epsilon e_n-y|^{n+2s}}\, dy \le C\,
\int_{\Omega\setminus B_\delta}\frac{dy}{|\epsilon e_n-y|^{n+2s}}\le
C\,\int_{\R^n\setminus B_\delta}\frac{dy}{|y|^{n+2s}}\, dy = C\, \delta^{-2s},
\end{equation}
up to renaming the constants.

On the othe hand, we have that
\begin{equation}\label{conto3}
\int_{\Omega}\frac{dy}{|\epsilon e_n-y|^{n+2s}}\ge
\int_{B_\delta\cap\Omega}\frac{dy}{|\epsilon e_n-y|^{n+2s}}=\Sigma_\epsilon.
\end{equation}

Finally, we observe that
\begin{equation}\begin{split}\label{conto5}
\epsilon^{2s}\,\Sigma_\epsilon =&
\epsilon^{2s}\, \int_{B_\delta\cap\Omega}\frac{dy}{|\epsilon e_n-y|^{n+2s}} \\
=&
\int_{B_{\delta/\epsilon}\cap(\epsilon^{-1}\Omega)}\frac{dz}{|e_n-z|^{n+2s}}\\
\ge & \int_{B_{\delta/\epsilon}\cap K}\frac{dz}{|e_n-z|^{n+2s}}\\
=: & \kappa,
\end{split}\end{equation}
where we have used the change of variable~$y=\epsilon z$ and~\eqref{kappa}.

Hence, using the second condition in~\eqref{over}
and putting together~\eqref{conto1}, \eqref{conto2}, \eqref{conto3}
and~\eqref{conto5}, we obtain
\begin{eqnarray*}
0 &=& \int_{\Omega}\frac{dy}{|\epsilon e_n-y|^{n+2s}}-
\int_{\Omega}\frac{u(\epsilon e_n)-u(x)}{|\epsilon e_n-y|^{n+2s}}\, dy \\
&=& \int_{\Omega}\frac{dy}{|\epsilon e_n-y|^{n+2s}}-
\int_{\Omega\cap B_\delta}\frac{u(\epsilon e_n)-u(x)}{|\epsilon e_n-y|^{n+2s}}\, dy
- \int_{\Omega\setminus B_\delta}\frac{u(\epsilon e_n)-u(x)}{|\epsilon e_n-y|^{n+2s}}\, dy \\
& \ge & \Sigma_\epsilon  -\frac{1}{2}\Sigma_\epsilon -C\, \delta^{-2s}\\
& =& \frac{1}{2}\Sigma_\epsilon -C\,\delta^{-2s}\\
&=& \epsilon^{-2s}\left(
\frac{\epsilon^{2s}}{2}\Sigma_\epsilon- C\, \epsilon^{2s}\, \delta^{-2s}\right)\\
&\ge & \epsilon^{-2s}\left(
\frac{\kappa}{2}- C\, \epsilon^{2s}\, \delta^{-2s}\right)>0
\end{eqnarray*}
if~$\epsilon$ is sufficiently small.
This gives a contradiction and concludes the proof.
\end{proof}

\section{Comparison with previous works}\label{S1}

In this last section we compare our new Neumann nonlocal conditions with the previous works in the literature that also deal with Neumann-type conditions for the fractional Laplacian $(-\Delta)^s$ (or related operators).

The idea of \cite{BBC,CK} (and also \cite{CERW,CERW2,CERW3}) is to consider the \emph{regional} fractional Laplacian, associated to the Dirichlet form
\begin{equation}\label{regional}
c_{n,s}\int_\Omega\int_\Omega\frac{\bigl(u(x)-u(y)\bigr)\bigl(v(x)-v(y)\bigr)}{|x-y|^{n+2s}}\,dx\,dy.
\end{equation}
This operator corresponds to a censored process, i.e., a process whose
jumps are restricted to be in $\Omega$.
The operator can be defined in general domains $\Omega$,
and seems to give a natural analogue of homogeneous Neumann condition.
However, no nonhomogeneous Neumann conditions can be considered with this model, and the operator depends on the domain $\Omega$.

On the other hand, in \cite{BCGJ,BGJ} the usual diffusion associated to the fractional Laplacian \eqref{operator} was considered inside $\Omega$, and thus the ``particle'' can jump outside $\Omega$.
When it jumps outside $\Omega$, then it is ``reflected'' or ``projected'' inside $\Omega$ in a deterministic way.
Of course, different types of reflections or projections lead to different Neumann conditions.
To appropriately define these reflections, some assumptions on the domain $\Omega$ (like smoothness or convexity) need to be done.
In contrast with the regional fractional Laplacian, this problem does not have a variational formulation and everything is done in the context of viscosity solutions.

In \cite{Grubb2} a different Neumann problem for the fractional Laplacian was considered.
Solutions to this type of Neumann problems are ``large solutions'', in the sense that they are not bounded in a neighborhood of $\partial\Omega$.
More precisely, it is proved in \cite{Grubb2} that the following problem is well-posed
\[\left\{ \begin{array}{rcll}
(-\Delta)^su&=&f&\textrm{in }\Omega \\
u&=&0&\textrm{in }\R^n\setminus\Omega\\
\partial_\nu\bigl(u/d^{s-1}\bigr)&=&g&\textrm{on }\partial\Omega,
\end{array}\right.\]
where $d(x)$ is the distance to $\partial\Omega$.

Finally, in \cite{spectral,spectral2} homogeneous Neumann problems for the \emph{spectral} fractional Laplacian were studied.
The operator in this case is defined via the eigenfunctions of the Laplacian $-\Delta$ in $\Omega$ with Neumann boundary condition $\partial_\nu u=0$ on $\partial\Omega$.

With respect to the existing literature,
the new Neumann problems \eqref{eq-intro} and \eqref{eq-heat-intro} that we present here have the following advantages:
\begin{itemize}
\item The equation satisfied inside $\Omega$ does not depend on anything
(domain, right hand side, etc). Notice that the operator in~\eqref{operator}
does not depend on the domain~$\Omega$, while for instance the
regional fractional Laplacian defined in~\eqref{regional} depends on~$\Omega$.
\item The problem can be formulated in general domains, including nonsmooth or even unbounded ones.
\item The problem has a variational structure. For instance,
solutions to the elliptic problem \eqref{eq-intro} can be found as critical points of the functional
\[\mathcal E(u)=\frac{c_{n,s}}{4}\int_{\R^{2n}\setminus(\mathcal C\Omega)^2}\frac{|u(x)-u(y)|^2}{|x-y|^{n+2s}}\,dx\,dy-\int_\Omega fu.\]
We notice that the variational formulation of the problem is the
analogue of the case $s=1$.
Also, this allows us to easily prove existence of solutions (whenever the compatibility condition $\int_\Omega f=0$ is satisfied).
\item Solutions to the fractional heat equation \eqref{eq-heat-intro} possess natural properties like conservation of mass inside $\Omega$ or convergence to a constant as $t\rightarrow+\infty$.
\item Our probabilistic interpretation allows us to formulate problems with nonhomogeneous Neumann conditions $\mathcal N_s u=g$ in $\R^n\setminus\overline{\Omega}$, or with mixed Dirichlet and Neumann conditions.
\item The formulation of nonlinear equations like $(-\Delta)^su=f(u)$ in $\Omega$ with Neumann conditions is also clear.
\end{itemize}

\appendix

\section*{Proof of Theorems~\ref{EX} and~\ref{EVAL}
with a functional analytic notation}

As anticipated in the footnote of page~\pageref{FT1},
we provide this appendix in order to satisfy the reader
who wish to prove Theorems~\ref{EX} and~\ref{EVAL}
by keeping the distinction between a function defined in the whole
of~$\R^n$ and its restriction to the domain~$\Omega$. For this scope,
we will use the notation of denoting~$r^+ u$ and~$r^- u$
the restriction of~$u$ to~$\Omega$ and~$\R^n\setminus\Omega$,
respectively. Notice that, in this notation, we have that~$u:\R^n\to\R$,
but $r^+ u:\Omega\to\R$ and~$r^- u:\R^n\setminus\Omega\to\R$.

\begin{proof}[Proof of Theorem~\ref{EX}] One can reduce to the case
$g\equiv0$. By the Riesz representation theorem,
given~$h\in L^2(\Omega)$, one finds~$v:=T_o h\in H^s_{\Omega,g}$
that is a weak solution of
$$r^+\big(  (-\Delta)^s v+v\big)=h,$$
with~$r^- \mathcal N_sv=0$.

Notice that~$T_o : L^2(\Omega)\to H^s_{\Omega,g}$.
We also define by~$T:L^2(\Omega)\to L^2(\Omega)$
the restriction operator of $T_o$, that is~$Th:= r^+ T_o h$.
One sees that~$T$ is compact and self-adjoint.
By construction~$r^- \mathcal N_s T_o h=0$ and
$$ h=r^+\big(  (-\Delta)^s T_o h+T_o h\big) =r^+(-\Delta)^s T_o h + Th,$$
that is
$$ r^- \mathcal N_s T_o=0
\ {\mbox{ and }} \ Id-T = r^+(-\Delta)^s T_o.$$
Therefore, by Lemma \ref{max-prin},
\begin{eqnarray*}
{Ker}(Id-T) &=& \{ h\in L^2(\Omega) {\mbox{ s.t. }} r^+(-\Delta)^s T_o h=0\}
\\ &=& \{ h\in L^2(\Omega) {\mbox{ s.t. }} r^+(-\Delta)^s T_o h=0
{\mbox{ and }} r^- \mathcal N_s T_o h=0\}
\\ &=& \{ h\in L^2(\Omega) {\mbox{ s.t. $T_o h$ is constant}} \}
\\ &=& \{ h\in L^2(\Omega) {\mbox{ s.t. $h$ is constant}} \}.
\end{eqnarray*}
{F}rom the Fredholm Alternative, we conclude that~$ {Im} (Id-T)$
is the space of functions in~$L^2(\Omega)$
that are orthogonal to constants.
\end{proof}

\begin{proof}[Proof of Theorem~\ref{EVAL}]
We define
$$ L^2_0(\Omega):=\left\{u\in L^2(\Omega) \ :\ \int_{\Omega}u =0\right\}.$$
By Theorem~\ref{EX}, for any~$f\in L^2_0(\Omega)$
one finds~$v:=T_o f\in H^s_{\Omega,g}$
that is a weak solution of~$r^+ (-\Delta)^s v=f$,
with~$r^-\mathcal N_s v=0$ and zero average in~$\Omega$.
We also define~$T$ to be the restriction of~$T_o$, that is~$Tf:=r^+ T_o f$.
The operator $T$ is compact and self-adjoint in~$L^2_0(\Omega)$.
Thus, by the spectral theorem there exists a sequence of eigenvalues
$\{\mu_i\}_{i\ge2}$ of~$T$,
and its corresponding eigenfunctions~$\{e_i\}_{i\ge2}$ are a
complete orthogonal system in~$L^2_0(\Omega)$.

Notice that~$r^-\mathcal N_s T_o e_i$, which gives, for every~$x\in\R^n\setminus\Omega$,
$$ T_o e_i(x)\,\int_\Omega \frac{dy}{|x-y|^{n+2s}}=
\int_\Omega \frac{r^+ T_oe_i(y)}{|x-y|^{n+2s}}\,dy
=\int_\Omega \frac{T e_i(y)}{|x-y|^{n+2s}}\,dy=
\mu_i \int_\Omega \frac{ e_i(y)}{|x-y|^{n+2s}}\,dy.$$
This gives that
$$ \mu_i\ne0.$$
Indeed, otherwise we would have that~$r^-T_o e_i=0$. Since also
$$ 0=\mu_i e_i = Te_i =r^+ T_o e_i,$$
we would get that~$T_o e_i=0$ and thus~$0=(-\Delta)^sT_o e_i=e_i$,
which is impossible.

As a consequence, we can define~$ \lambda_i := \mu_i^{-1}$,
and~$u_i:=T_o e_i$.

Then
$$ r^+ u_i=r^+ T_oe_i= T e_i=\mu_i e_i$$
thus~$\{r^+ u_i\}_{i\ge2}$ are a
complete orthogonal system in~$L^2_0(\Omega)$, since so are~$\{e_i\}_{i\ge2}$.

Furthermore, $r^+ (-\Delta)^s u_i = r^+ (-\Delta)^s T_o e_i=
e_i= \mu_i^{-1} r^+ u_i=r^+ \lambda_i u_i$.
\end{proof}


\begin{thebibliography}{00}

\bibitem{BCGJ} G. Barles, E. Chasseigne, C. Georgelin, E. Jakobsen, \emph{On Neumann type problems for nonlocal equations in a half space}, Trans. Amer. Math. Soc. 366 (2014), 4873-4917.

\bibitem{BCI2} G. Barles, E. Chasseigne, C. Imbert, \emph{The Dirichlet problem for second-order elliptic integro-differential equations}, Indiana Univ. Math. J. 57 (2008), 213-146.

\bibitem{BGJ} G. Barles, C. Georgelin, E. Jakobsen, \emph{On Neumann and oblique derivatives boundary conditions for nonlocal elliptic equations}, J. Differential Equations 256 (2014), 1368-1394.

\bibitem{BBC} K. Bogdan, K. Burdzy, Z.-Q. Chen, \emph{Censored stable processes}, Probab. Theory Relat. Fields 127 (2003), 89-152.

\bibitem{BBM} J. Bourgain, H. Brezis, P. Mironescu, \emph{Another look at Sobolev spaces}, in: J.L. Menaldi, E. Rofman, A. Sulem (Eds.), Optimal Control and Partial Differential Equations, IOS Press, Amsterdam, 2001, pp. 439--455. A volume
in honor of A. Bensoussan's 60th birthday.

\bibitem{CRS} L. Caffarelli, J.M. Roquejoffre, O. Savin, \emph{Nonlocal minimal surfaces}, Comm. Pure Appl. Math. 63 (2010), 1111-1144.

\bibitem{CS} L. Caffarelli, L. Silvestre, \emph{Regularity theory for fully nonlinear integro-differential equations}, Comm. Pure Appl. Math. 62 (2009), 597-638.

\bibitem{CK} Z.-Q. Chen, P. Kim, \emph{Green function estimate for censored stable processes}, Probab. Theory Relat. Fields 124 (2002), 595--610.

\bibitem{CERW} C. Cortazar, M. Elgueta, J. Rossi, N. Wolanski, \emph{Boundary fluxes for nonlocal diffusion}, J. Differential Equations 234 (2007), 360-390.

\bibitem{CERW2} C. Cortazar, M. Elgueta, J. Rossi, N. Wolanski, \emph{How to approximate the heat equation with Neumann boundary conditions by nonlocal diffusion problems}, Arch. Rat. Mech. Anal. 187 (2008), 137-156.

\bibitem{CERW3} C. Cortazar, M. Elgueta, J. Rossi, N. Wolanski, \emph{Asymptotic behavior for nonlocal diffusion equations}, J. Math. Pures Appl. 86 (2006), 271-291.

\bibitem{DPV} E. Di Nezza, G. Palatucci, E. Valdinoci, \emph{Hitchhiker's guide to the fractional Sobolev spaces}, Bull. Sci. Math. 136 (2012), 521-573.

\bibitem{FKV} M. Felsinger, M. Kassmann, P. Voigt, \emph{The Dirichlet problem for nonlocal operators}, to appear in Math. Z.

\bibitem{Grubb} G. Grubb, \emph{Fractional Laplacians on domains, a development of H\"ormander's theory of $\mu$-transmission pseudodifferential operators}, to appear in Adv. Math.

\bibitem{Grubb2} G. Grubb, \emph{Local and nonlocal boundary conditions for $\mu$-transmission and fractional order elliptic pseudodifferential operators}, to appear in Anal. PDE.

\bibitem{Jost} J. Jost, \emph{Partial differential equations},
Springer, New York, 2013.

\bibitem{Landkov} N. S. Landkof, \emph{Foundations of Modern Potential Theory}, Springer, New York, 1972.

\bibitem{spectral} E. Montefusco, B. Pellacci, G. Verzini, \emph{Fractional diffusion with Neumann boundary conditions: the logistic equation}, Disc. Cont. Dyn. Syst. Ser. B 18 (2013), 2175-2202.

\bibitem{RS-Dir} X. Ros-Oton, J. Serra, \emph{The Dirichlet problem for the fractional Laplacian: regularity up to the boundary},
J. Math. Pures Appl. 101 (2014), 275-302.

\bibitem{SV-weak} R. Servadei, E. Valdinoci, \emph{Weak and viscosity solutions of the fractional Laplace equation},
Publ. Mat. 58 (2014), 133-154.

\bibitem{S-obst} L. Silvestre, \emph{Regularity of the obstacle problem for a fractional power of the Laplace operator}, Comm. Pure Appl. Math. 60 (2007), 67-112.
    
\bibitem{spectral2} P. Stinga, B. Volzone, \emph{Fractional semilinear Neumann problems arising from a fractional Keller--Segel model}, preprint arXiv (June 2014).

\bibitem{V-SEMA} E. Valdinoci, \emph{From the long jump random walk to the fractional Laplacian}, Bol. Soc. Esp. Mat. Apl.
S$\vec{\text{e}}$MA 49 (2009), 33-44.


\end{thebibliography}
\end{document}